\documentclass[11pt]{amsart}

\usepackage{amssymb}
\usepackage[all]{xypic}
\usepackage{mathrsfs}
\usepackage{amsmath}
\usepackage{amsthm}
\usepackage{graphicx}

\def\scr#1{{\scriptstyle{#1}}}
\def\dd{\mathinner{\mkern1mu\raise8pt\hbox{.}\mkern2mu
\raise4pt\hbox{.}\mkern2mu\hbox{.}\mkern1mu}}
\def\rdd{\mathinner{\mkern1mu\hbox{.}\mkern2mu
\raise4pt\hbox{.}\mkern2mu\raise8pt\hbox{.}\mkern1mu}}

\newcommand{\mo}{\mathrm{mod}\,}
\newcommand{\ind}{\mathrm{ind}\,}
\newcommand{\Hom}{\mathrm{Hom}}
\newcommand{\Tr}{\mathrm{Tr}}

\newcommand{\coh}{\mathrm{coh}}
\newcommand{\ac}{\mathrm{ac}}
\newcommand{\End}{\mathrm{End}}
\newcommand{\Ext}{\mathrm{Ext}}
\newcommand{\rad}{\mathrm{rad}}
\newcommand{\add}{\mathrm{add}}
\newcommand{\Tor}{\mathrm{Tor}}
\newcommand{\ann}{\mathrm{ann}}

\newcommand{\tom}{\mathrm{top}}
\newcommand{\soc}{\mathrm{soc}}

\newcommand{\ql}{\mathrm{ql}}
\newcommand{\pd}{\mathrm{pd}}
\newcommand{\id}{\mathrm{id}}
\newcommand{\gldim}{\mathrm{gl.dim\,}}

\newtheorem{thm}{Theorem}[section]
\newtheorem{prop}[thm]{Proposition}

\newtheorem{cor}[thm]{Corollary}

\theoremstyle{definition}
\newtheorem{ex}[thm]{Example}

\newcommand{\sC}{\mathscr{C}}
\newcommand{\sO}{\mathscr{O}}

\newcommand{\sD}{\mathscr{D}}
\newcommand{\sT}{\mathscr{T}}
\newcommand{\sE}{\mathscr{E}}

\newcommand{\sX}{\mathscr{X}}
\newcommand{\sY}{\mathscr{Y}}
\newcommand{\sH}{\mathscr{H}}

\newcommand{\cP}{\mathcal{P}}
\newcommand{\cQ}{\mathcal{Q}}
\newcommand{\cT}{\mathcal{T}}

\begin{document}

\title[The structure of generalized standard Auslander-Reiten components]{The structure and homological properties of generalized standard Auslander-Reiten components}
%
\author[Malicki]{Piotr Malicki}
\address{Faculty of Mathematics and Computer Science, Nicolaus Copernicus University, Chopina 12/18, 87-100 Toru\'n, Poland}
\email{pmalicki@mat.umk.pl}

\thanks{This work was completed with the support of the Research Grant DEC-2011/02/A/ST1/00216 of~the National Science Center Poland.}
%
\author[Skowro\'nski]{Andrzej Skowro\'nski}
\address{Faculty of Mathematics and Computer Science, Nicolaus Copernicus University, Chopina 12/18, 87-100 Toru\'n, Poland}
\email{skowron@mat.umk.pl}


\dedicatory{Dedicated to Jos\'e Antonio de la Pe\~na on the occasion of his sixtieth birthday}


\baselineskip 15pt

\begin{abstract}
We describe the structure and homological properties of arbitrary generalized standard Auslander-Reiten components of artin algebras. In particular, we prove
that for all but finitely many indecomposable modules in such components the Euler characteristic is defined and nonnegative. Further, we provide a~handy criterion
for an infinite Auslander-Reiten component of an artin algebra to be generalized standard.
We solve also the long standing open problem concerning the structure of artin algebras admitting a~separating family of Auslander-Reiten components.
\end{abstract}
\maketitle


\section{Introduction and the main results} \label{intro}

Throughout the paper, by an algebra we mean a~basic indecomposable artin algebra over a~commutative artin ring $K$.
For an algebra $A$, we denote by $\mo A$ the category of finitely generated right $A$-modules, by $\ind A$ the full subcategory of $\mo A$ formed
by the indecomposable modules, and by $D$ the standard duality on $\mo A$.
The radical $\rad_A$ of $\mo A$ is the ideal generated by all nonisomorphisms between modules in $\ind A$. Then the infinite radical $\rad^{\infty}_A$ of $\mo A$ is the intersection
of all powers $\rad^i_A$, $i\geq 1$, of $\rad_A$.
By a~result of Auslander \cite{Au}, $\rad_A^{\infty}=0$ if and only if $A$ is of finite representation type, that is, there are in $\mo A$ only finitely many indecomposable modules
up to isomorphism. On the other hand, if $A$ is of infinite representation type then $(\rad_A^{\infty})^2\neq 0$, by a~result proved in \cite{CMMS}.

An important combinatorial and homological invariant of the module category $\mo A$ of an algebra $A$ is its Auslander-Reiten quiver $\Gamma_A$.
Recall that $\Gamma_A$ is a~valued translation quiver whose vertices are the isomorphism classes $\{X\}$ of modules $X$ in $\ind A$, the arrows correspond
to irreducible homomorphisms between modules in $\ind A$, and the translation is the Auslander-Reiten translation $\tau_A=D\Tr$.
We shall not distinguish between a~module $X$ in $\ind A$ and the vertex $\{X\}$ of $\Gamma_A$.
Moreover, by a~component of $\Gamma_A$ we mean a~connected component of the quiver $\Gamma_A$. Frequently, algebras can be recovered from the shapes of the components
of their Auslander-Reiten quiver. Further, very often the behavior of components of the Auslander-Reiten quiver $\Gamma_A$ of an algebra $A$ in the module category $\mo A$
leads to essential information on $A$, allowing to determine $A$ and $\mo A$ completely.
For a~component $\sC$ of $\Gamma_A$, we denote by $\ann_A(\sC)$ the annihilator of $\sC$ in $A$, that is, the intersection of the annihilators
$\{a\in A\mid Ma=0\}$ of all modules $M$ in $\sC$, and by $B(\sC)$ the quotient algebra $A/\ann_A(\sC)$, called the faithful algebra of $\sC$.
We note that $\sC$ is a~faithful component of $\Gamma_{B(\sC)}$.

In the representation theory of finite-dimensional algebras over an algebraically closed field $k$ a~prominent role is played by the standard
Auslander-Reiten components (see \cite{BGRS, BoG, Ke, Ri1, Ri2, S3, S4} for some results characterizing tame algebras via standardness of their Auslander-Reiten components).
Recall that, following \cite{BoG, Ri1}, a~component $\sC$ in the Auslander-Reiten quiver $\Gamma_{\Lambda}$ of a~finite-dimensional algebra $\Lambda$
over $k$ is called standard if the full subcategory of $\mo\Lambda$ formed by all modules from $\sC$ is equivalent to the mesh-category $k(\sC)$ of $\sC$.
In particular, one knows \cite{BoG, BrG} that, for $\Lambda$ of finite representation type,
$\Gamma_{\Lambda}$ is standard if and only if $\Lambda$ admits a~simply connected Galois covering.
Moreover, long time ago Ringel asked \cite[Problem 3]{Ri-86} if any standard regular component of the Auslander-Reiten quiver of an algebra over
an algebraically closed field is either a~stable tube or of the form ${\Bbb Z}\Delta$ for a~finite acyclic quiver $\Delta$.
This was shown to be true, also in the wider context of artin algebras \cite[Corollaries 2.4 and 2.5]{S2}.

The second named author proposed in \cite{S2} a~natural generalization of the concept of standard component, called generalized standard component, which is simpler and makes sense for
any artin algebra. Namely, a~component $\sC$ of the Auslander-Reiten quiver $\Gamma_A$ of an algebra $A$ is called generalized standard if $\rad_A^{\infty}(X,Y)=0$ for all modules
$X$ and $Y$ from $\sC$. It follows from general theory \cite{ARS} that every nonzero nonisomorphism $f: X\to Y$ with $X$ and $Y$ in a~generalized standard component $\sC$ of $\Gamma_A$
is a~finite sum of compositions of irreducible homomorphisms between indecomposable modules from $\sC$. Moreover, the additive category $\add(\sC)$ of a~generalized standard component
$\sC$ of $\Gamma_A$ is closed under extensions in $\mo A$, and this allows to describe the degeneration-like orders for modules in $\add(\sC)$ with the same composition factors (see \cite{SZ}).
We also mention that $\sC$ is a~generalized standard component of $\Gamma_A$ if and only if $\sC$ is a~generalized standard component of $\Gamma_{B(\sC)}$.
The Auslander-Reiten quiver of every algebra of finite representation type is generalized standard. Further, Liu proved in \cite{L4} that every
standard component of an Auslander-Reiten quiver of a~finite-dimensional algebra over an algebraically closed field is generalized standard.
The converse implication is not true, because there are nonstandard Auslander-Reiten quivers for some algebras of finite representation type
over algebraically closed fields of characteristic $2$ \cite{Rie}. The following results show that existence of a~generalized standard component in the
Auslander-Reiten quiver $\Gamma_A$ of an algebra $A$ may determine the algebra $A$.
For example, it is known that an algebra $A$ is a~tilted algebra (respectively, double tilted algebra, generalized double tilted algebra) if and only if
$\Gamma_A$ admits a~faithful generalized standard component $\sC$ with a~section \cite{L3, S1} (respectively, double section \cite{RS2},
multisection \cite{RS3}). Moreover, by results established in \cite{SY1, SY2}, every self-injective algebra $A$ for which $\Gamma_A$ admits an acyclic generalized standard
component is a~socle deformation of the orbit algebra $\widehat{B}/G$ of the repetitive category $\widehat{B}$ of a~tilted algebra $B$ of infinite representation type
and an infinite cyclic automorphism group $G$ of $\widehat{B}$. We also mention that distinguished  classes of generalized standard components are formed by the separating families
of tubes of quasitilted algebras of canonical type \cite{LS}, or more generally, the separating families of generalized multicoils of generalized multicoil algebras \cite{MS2}.

We are concerned with the structure of an arbitrary generalized standard component of an Auslander-Reiten quiver. It has been proved in \cite{S2} that every generalized
standard component $\sC$ of the Auslander-Reiten quiver $\Gamma_A$ of an algebra $A$ is almost periodic, that is, all but finitely many $\tau_A$-orbits in $\sC$ are periodic.
This implies that $\sC$ may contain at most finitely many indecomposable modules of any given length. The acyclic generalized standard components were described
completely in \cite{S1}. In particular, the acyclic generalized standard semiregular components of an Auslander-Reiten quiver $\Gamma_A$ are the connecting components
of quotient tilted algebras of $A$ (see \cite{L3, S2}). On the other hand, the description of generalized standard components with oriented cycles is an
exciting but difficult problem. Namely, it was shown in \cite{S6} that every finite-dimensional algebra $\Lambda$ over a~field $K$ is a~quotient algebra of an algebra
$A$ with $\Gamma_A$ having a~faithful generalized standard stable tube (see also \cite{LeSk, MS2-5}).
In general, one needs some extra information concerning interaction of a~given component $\sC$
of $\Gamma_A$ with other components of $\Gamma_A$. For example, it was shown in \cite{JMS} that if a~semiregular component $\sC$ of $\Gamma_A$ is without external short paths
(in the sense of \cite{RS1}) in $\mo A$ then $\sC$ is generalized standard and a~component of a~quasitilted quotient algebra of $A$.

The first main result of the paper provides a~handy criterion for an infinite component of an Auslander-Reiten quiver $\Gamma_A$ to be generalized standard.
Let $\sC$ be an almost periodic component of the Auslander-Reiten quiver $\Gamma_A$ of an algebra $A$. Following \cite{MS1} an indecomposable projective module
$P$ in $\sC$ is said to be right coherent if there is an infinite sectional path
$P~=~X_1 \longrightarrow X_2 \longrightarrow \cdots \longrightarrow X_i\longrightarrow X_{i+1} \longrightarrow X_{i+2} \longrightarrow \cdots$
(that is, $X_i\neq \tau_AX_{i+2}$ for all $i\geq 0$) in $\sC$. Dually, an indecomposable injective module $Q$ in $\sC$ is said to be left coherent if
there is an infinite sectional path
$\cdots \longrightarrow Y_{j+2} \longrightarrow Y_{j+1} \longrightarrow Y_j \longrightarrow \cdots \longrightarrow Y_2 \longrightarrow Y_1 = Q$
(that is, $Y_{j+2}\neq \tau_AY_{j}$ for all $j\geq 0$) in $\sC$.
We denote by $P_{\sC}$ (respectively, $P_{\sC}^{\coh}$) the direct sum of all indecomposable projective (respectively, right coherent projective)
modules in $\sC$,
and by $Q_{\sC}$ (respectively, $Q_{\sC}^{\coh}$) the direct sum of all indecomposable injective (respectively, left coherent injective) modules in $\sC$.
Applying \cite{L2}, we may also distinguish in $\sC$ a~left stable acyclic module $M_{\sC}^{(l)}$ of $\sC$, being the direct sum of indecomposable modules
forming left sections of the left stable acyclic components of $\sC$, and a~right stable acyclic module $M_{\sC}^{(r)}$ of $\sC$, being the direct sum of indecomposable modules
forming right sections of the right stable acyclic components of $\sC$.
We may also assume that
there is no path in $\sC$ from a~direct summand of $M_{\sC}^{(r)}$ to a~direct summand of $M_{\sC}^{(l)}$.
Further, applying \cite{MS1}, we may define a~tubular module $M_{\sC}^{(t)}$ of $\sC$,
being the direct sum of indecomposable modules forming the mouth of all stable tubes involved in constructing the maximal cyclic coherent full translation
subquivers of $\sC$. We refer to Section \ref{sec2} for details.

We are now in position to formulate the first main result of the paper.
\begin{thm} \label{thm1}
Let $A$ be an algebra and $\sC$ be an infinite component of $\Gamma_A$. The following statements are equivalent:
\begin{enumerate}
\renewcommand{\labelenumi}{\rm(\roman{enumi})}
\item $\sC$ is generalized standard.
\item $\sC$ is almost periodic and the following vanishing conditions hold
\[\Hom_A(P_{\sC}\oplus M_{\sC}^{(t)}\oplus M_{\sC}^{(r)},M_{\sC}^{(l)})=0, \Hom_A(M_{\sC}^{(r)},Q_{\sC}\oplus M_{\sC}^{(t)}\oplus M_{\sC}^{(l)})=0,\hspace{2mm} \]
\[ \Hom_A(M_{\sC}^{(l)},\tau_AM_{\sC}^{(l)})=0 , \rad_A(M_{\sC}^{(t)},M_{\sC}^{(t)})=0, \Hom_A(\tau_A^{-1}M_{\sC}^{(r)},M_{\sC}^{(r)})=0,  \]
\[\Hom_A(P_{\sC},\soc(Q_{\sC}^{\coh})\oplus M_{\sC}^{(t)})=0, \Hom_A(\tom(P_{\sC}^{\coh})\oplus M_{\sC}^{(t)},Q_{\sC})=0.\hspace{11mm}\]
\end{enumerate}
\end{thm}
The second main result of the paper describes the structure of an arbitrary infinite generalized standard component of an Auslander-Reiten quiver.
\begin{thm} \label{thm2}
Let $A$ be an algebra and $\sC$ an infinite generalized standard component of $\Gamma_A$. Then there are quotient algebras
$A_{\sC}^{(lt)}$, $A_{\sC}^{(c)}$, $A_{\sC}^{(rt)}$
of $A$ such that the following statements hold:
\begin{enumerate}
\renewcommand{\labelenumi}{\rm(\roman{enumi})}
\item $A_{\sC}^{(lt)}=A_{1}^{(lt)} \times\cdots\times A_{m}^{(lt)},$ where
\begin{enumerate}
\item For each $i\in\{1,\ldots,m\}$, $A_{i}^{(lt)}$ is a~tilted algebra of the form $\End_{H^{(l)}_i}(T^{(l)}_i)$, for a~hereditary algebra $H^{(l)}_i$
and a~tilting module $T^{(l)}_i$ in $\mo H^{(l)}_i$ without indecomposable preinjective direct summands.
\item For each $i\in\{1,\ldots,m\}$, the image ${\sC}^{(l)}_i$ of the preinjective component ${\cQ}(H^{(l)}_i)$ of $\Gamma_{H^{(l)}_i}$ via the functor
$\Hom_{H^{(l)}_i}(T^{(l)}_i,-): \mo H^{(l)}_i \to \mo A^{(lt)}_i$ is an acyclic full translation subquiver of $\sC$ which is closed under predecessors.
\end{enumerate}
\item $A_{\sC}^{(rt)}=A_{1}^{(rt)} \times\cdots\times A_{n}^{(rt)},$ where
\begin{enumerate}
\item For each $j\in\{1,\ldots,n\}$, $A_{j}^{(rt)}$ is a~tilted algebra of the form $\End_{H^{(r)}_j}(T^{(r)}_j)$, for a~hereditary algebra $H^{(r)}_j$
and a~tilting module $T^{(r)}_j$ in $\mo H^{(r)}_j$ without indecomposable postprojective direct summands.
\item For each $j\in\{1,\ldots,n\}$, the image ${\sC}^{(r)}_j$ of the postprojective component ${\cP}(H^{(r)}_j)$ of $\Gamma_{H^{(r)}_j}$ via the functor
$\Ext^1_{H^{(r)}_j}(T^{(r)}_j,-): \mo H^{(r)}_j$ $\to \mo A^{(rt)}_j$ is an acyclic full translation subquiver of $\sC$ which is closed under successors.
\end{enumerate}
\item $A_{\sC}^{(c)}=A_{1}^{(c)} \times\cdots\times A_{p}^{(c)},$ where
\begin{enumerate}
\item For each $k\in\{1,\ldots,p\}$, $A_{k}^{(c)}$ is a~generalized multicoil enlargement of an algebra $B_{k}^{(c)}$ with a~faithful family
${\cT}^{B_{k}^{(c)}}$ of pairwise orthogonal generalized standard stable tubes in $\Gamma_{B_{k}^{(c)}}$.
\item For each $k\in\{1,\ldots,p\}$, $\Gamma_{A_{k}^{(c)}}$ admits a~generalized multicoil ${\sC}_{k}^{(c)}$, obtained from a~finite number of stable tubes
in ${\cT}^{B_{k}^{(c)}}$ by translation quiver admissible operations corresponding to the algebra admissible operations leading from $B_{k}^{(c)}$ to $A_{k}^{(c)}$,
and the cyclic part $_c{\sC}_{k}^{(c)}$ of ${\sC}_{k}^{(c)}$ is a~full translation subquiver of $\sC$.
\end{enumerate}
\item The translation quivers ${\sC}^{(l)} = {\sC}_{1}^{(l)}\cup \ldots\cup {\sC}_{m}^{(l)}$ and ${\sC}^{(r)} = {\sC}_{1}^{(r)}\cup \ldots\cup {\sC}_{n}^{(r)}$
are disjoint and their union ${\sC}^{(l)}\cup {\sC}^{(r)}$ contains all but finitely many acyclic indecomposable modules of $\sC$.
\item The translation quivers ${\sC}_{1}^{(c)}, \ldots, {\sC}_{p}^{(c)}$ are pairwise disjoint and their union
${\sC}^{(c)} = {\sC}_{1}^{(c)}\cup \ldots\cup {\sC}_{p}^{(c)}$ contains all but finitely many indecomposable modules of the cyclic part $_c\sC$ of $\sC$.
\end{enumerate}
\end{thm}

The algebra $A_{\sC}^{(lt)}$ is said to be the \emph{left tilted algebra} of $\sC$ and the algebra $A_{\sC}^{(rt)}$ is said to be the \emph{right tilted algebra} of $\sC$.
Further, the algebra $A_{\sC}^{(c)}$ is said to be the \emph{coherent algebra} of $\sC$. We mention that $A_{\sC}^{(lt)}$ (respectively, $A_{\sC}^{(rt)}$)
is nontrivial provided $\sC$ admits left stable (respectively, right stable) acyclic part. Similarly, $A_{\sC}^{(c)}$ is nontrivial provided the cyclic part
$_c\sC$ of $\sC$ is infinite.

For an algebra $A$ and a~module $M$ in $\mo A$, we denote by $|M|$ the length of $M$ over the commutative artin ring $K$.

The following corollary is a~direct consequence of Theorem \ref{thm2} and \cite[Theorem 1.3]{MS3}.
\begin{cor} \label{cor3}
Let $A$ be an algebra and $\sC$ be a~generalized standard component of $\Gamma_A$. Then for all but finitely many indecomposable modules $M$ in $\sC$ we have
\begin{enumerate}
\renewcommand{\labelenumi}{\rm(\roman{enumi})}
\item $|\Ext_A^1(M,M)| \leq |\End_A(M)|$.
\item $\Ext_A^r(M,M) = 0$ for any $r\geq 2$.
\end{enumerate}
\end{cor}

Therefore, for all but finitely many modules $M$ in a~generalized standard component $\sC$ of an Auslander-Reiten quiver $\Gamma_A$, the
Euler characteristic
\[ \chi_A(M) = \sum_{i=0}^{\infty} (-1)^i |\Ext_A^i(M,M)| \]
is defined and nonnegative.

Recall that a~component $\sC$ of an Auslander-Reiten quiver $\Gamma_A$ is called regular if $\sC$ is without projective modules and injective modules.
It is known that every regular generalized standard component $\sC$ of $\Gamma_A$ is either a~stable tube or of the form ${\Bbb Z}\Delta$
for a~finite acyclic valued quiver $\Delta$ (see \cite[Corollary 2.4]{S2}).
In the later case, the faithful algebra $B(\sC)$ of $\sC$ is a~tilted algebra $\End_H(T)$ for a~wild hereditary algebra $H$
and a~regular tilting module $T$ in $\mo H$ (see \cite[Corollary 3.3]{S2}).
We refer to \cite{Ri1-2} (see also \cite[Section VIII.9]{SY4}) for the existence of acyclic regular generalized standard components.
On the other hand, the structure of the faithful algebra of a~generalized standard stable tube is still an open problem.
But the homological properties of indecomposable modules in generalized
standard stable tubes are described in \cite[Corollary 3.6]{S5}.

We obtain then the following consequence of these results and Corollary \ref{cor3}.
\begin{cor} \label{cor4}
Let $A$ be an algebra such that every component in $\Gamma_A$ is generalized standard.
Then for all but finitely many modules $M$ in $\ind A$ the Euler characteristic $\chi_A(M)$ is defined and
\[ \chi_A(M) = |\End_A(M)| - |\Ext_A^1(M,M)| \geq 0. \]
\end{cor}

A prominent role in the representation theory of algebras is played by the algebras with separating families of Auslander-Reiten components.
A~concept of a~separating family of tubes has been introduced by Ringel in \cite{Ri0, Ri1} who proved that they occur in the Auslander-Reiten quivers
of hereditary algebras of Euclidean type, tubular algebras, and canonical algebras.
In order to deal with wider classes of algebras, the following more general concept of a~separating family of Auslander-Reiten components
was proposed by Assem, the second named author and Tom\'e in \cite{AST} (see also \cite{MS2}).
A~family ${\sC}$ = $({\sC}_{i})_{i \in I}$ of components of the Auslander-Reiten quiver $\Gamma_A$ of an algebra $A$ is called separating in $\mo A$
if the components of $\Gamma_A$ split into three disjoint families ${\mathcal P}^A$, ${\sC}^A={\sC}$ and ${\mathcal Q}^A$ such that:
\begin{enumerate}
\item[(S1)] ${\sC}^A$ is a sincere family of pairwise orthogonal generalized standard components;
\item[(S2)] $\Hom_{A}({\mathcal Q}^A,{\mathcal P}^A) = 0$, $\Hom_{A}({\mathcal Q}^A,{\sC}^A)=0$, $\Hom_{A}({\sC}^A,{\mathcal P}^A) = 0$;
\item[(S3)] any homomorphism from ${\mathcal P}^A$ to ${\mathcal Q}^A$ in $\mo A$ factors through the additive category $\add({\sC}^A)$ of ${\sC}^A$.
\end{enumerate}
\noindent Then we say that ${\sC}^A$ separates ${\mathcal P}^A$ from ${\mathcal Q}^A$ and write
\[\Gamma_A={\mathcal P}^A \cup {\sC}^A \cup {\mathcal Q}^A.\]
We note that then ${\mathcal P}^A$ and ${\mathcal Q}^A$ are uniquely determined by ${\sC}^A$ (see \cite[(2.1)]{AST} or \cite[(3.1)]{Ri1}).
Moreover, we have $\ann_A({\sC}^A)=0$, so ${\sC}^A$ is a~faithful family of components of $\Gamma_A$.
We note that if $A$ is an algebra of finite representation type that ${\sC}^A = \Gamma_A$ is trivially a~unique separating component of $\Gamma_A$,
with ${\mathcal P}^A$ and ${\mathcal Q}^A$ being empty.
It is known that an algebra $A$ is
a~generalized double tilted algebra if and only if $\Gamma_A$ admits a~separating almost acyclic component \cite{RS3}.
In \cite{LP} Lenzing and de la Pe\~na proved that an Auslander-Reiten quiver $\Gamma_A$ admits a~separating family of stable tubes if and only if
$A$ is a~concealed canonical algebra. Moreover, by a~result proved in \cite{LS}, $\Gamma_A$ admits a~separating family of semiregular tubes if and only if
$A$ is a~quasitilted algebra of canonical type. This was extended in \cite{MS2} to the following result: the Auslander-Reiten quiver $\Gamma_A$
of an algebra $A$ admits a~separating family of almost cyclic coherent components if and only if $A$ is a~\emph{generalized multicoil algebra}, that is,
is a~generalized multicoil enlargement of a~product of quasitilted algebras of canonical type.
We refer to the survey article \cite{MS4} for more details on algebras with separating families of Auslander-Reiten components and their
representation theory.
%

The next aim is to describe the structure of the Auslander-Reiten quiver $\Gamma_A$ of an algebra $A$ with a~separating family of components.
We need some notation.

Let $H$ be a~hereditary algebra, $T$ a~tilting module in $\mo H$, and $B=\End_H(T)$ the associated tilted algebra. Then $\Gamma_B$ admits an acyclic component
${\sC}_T$ with the section $\Delta_T$ given by the images of the indecomposable injective modules in $\mo H$ via the functor $\Hom_H(T,-): \mo H \to \mo B$.
Moreover, $\Gamma_B$ has a~decomposition
\[\Gamma_B={\mathcal P}^B \cup {\sC}^B \cup {\mathcal Q}^B,\]
where ${\cP}^B$ is the disjoint union of all components of $\Gamma_B$ contained entirely in the torsion part $\sY(T)=\{Y\in\mo B \,|$ $\,\Tor_1^B(Y,T)=0\}$,
${\cQ}^B$ is the disjoint union of all components of $\Gamma_B$ contained entirely in the torsion-free part $\sX(T)=\{X\in\mo B \,|\,X\otimes_BT)=0\}$,
and ${\sC}^B = {\sC}_T$ separates ${\cP}^B$ from ${\cQ}^B$ (see \cite{HR, SY4}).

Let $\Lambda$ be a~quasitilted algebra of canonical type. Then $\Gamma_{\Lambda}$ has a~decomposition
\[\Gamma_{\Lambda}={\mathcal P}^{\Lambda} \cup {\sC}^{\Lambda} \cup {\mathcal Q}^{\Lambda},\]
where ${\sC}^{\Lambda}$ is a~family of semiregular tubes separating ${\mathcal P}^{\Lambda}$ from ${\mathcal Q}^{\Lambda}$ (see \cite{LS}).

The next theorem provides solution of the problem concerning the structure of artin algebras admitting a~separating family
of Auslander-Reiten components, initiated by Ringel \cite{Ri0, Ri1, Ri-86}.
\begin{thm}\label{thm7}
Let $A$ be an algebra with a separating family ${\sC}^A$ of components in $\Gamma_A$, and $\Gamma_A$=${\mathcal P}^A \cup {\sC}^A \cup {\mathcal Q}^A$
the associated decomposition of $\,\Gamma_A$. Then there exist quotient algebras $A^{(l)}$ and $A^{(r)}$ of $A$ such that the following statements hold.
\begin{enumerate}
\renewcommand{\labelenumi}{\rm(\roman{enumi})}
\item $A^{(l)}=A^{(l)}_1 \times\cdots\times A^{(l)}_m \times A^{(l)}_{m+1} \times\cdots\times A^{(l)}_{m+p},$  where
\begin{enumerate}
\item For each $i\in\{1,\ldots,m\}$, $A^{(l)}_i$ is a~tilted algebra of the form $\End_{H^{(l)}_i}(T^{(l)}_i)$ for a~hereditary algebra $H^{(l)}_i$
and a~tilting module $T^{(l)}_i$ in $\mo H^{(l)}_i$ without indecomposable preinjective  direct summands.
\item For each $i\in\{m+1,\ldots,m+p\}$, $A^{(l)}_i$ is a~quasitilted algebra of canonical type with a~separating family of coray tubes in $\Gamma_{A^{(l)}_i}$.
\end{enumerate}
\item $A^{(r)}=A^{(r)}_1 \times\cdots\times A^{(r)}_n \times A^{(r)}_{n+1} \times\cdots\times A^{(r)}_{n+q},$ where
\begin{enumerate}
\item For each $j\in\{1,\ldots,n\}$, $A^{(r)}_j$ is a~tilted algebra of the form $\End_{H^{(r)}_j}(T^{(r)}_j)$ for a~hereditary algebra $H^{(r)}_j$
and a~tilting module $T^{(r)}_j$ in $\mo H^{(r)}_j$ without indecomposable postprojective direct summands.
\item For each $j\in\{n+1,\ldots,n+q\}$, $A^{(r)}_j$ is a~quasitilted algebra of canonical type with a~separating family of ray tubes in $\Gamma_{A^{(r)}_j}$.
\end{enumerate}
\item  $\cP^A = \bigcup_{i=1}^{m+p} \cP^{A^{(l)}_i}$
and every component in $\cP^A$ is either a~postprojective component, a~ray tube, or obtained from a~component of the form ${\Bbb Z}{\Bbb A}_{\infty}$
by a~finite number (possibly zero) of ray insertions.
\item  $\cQ^A = \bigcup_{j=1}^{n+q} \cQ^{A^{(r)}_j}$
and every component in $\cQ^A$ is either a~preinjective component, a~coray tube, or obtained from a~component of the form ${\Bbb Z}{\Bbb A}_{\infty}$
by a~finite number (possibly zero) of coray insertions.
\end{enumerate}
\end{thm}

The proof of the next theorem applies the representation theory of generalized double tilted algebras developed by Reiten and Skowro\'nski
in \cite{RS2, RS3}. We also note that any algebra of finite representation type is a~generalized double tilted algebra in the sense of \cite{RS3}.
\begin{thm} \label{thm4}
Let $A$ be an algebra. The following statements are equivalent:
\begin{enumerate}
\renewcommand{\labelenumi}{\rm(\roman{enumi})}
\item $\Gamma_A$ admits a~finite separating family of components.
\item $\Gamma_A$ admits a~separating almost acyclic component.
\item $A$ is a~generalized double tilted algebra.
\end{enumerate}
\end{thm}

We note that the equivalence of (ii) and (iii) follows from \cite[Theorem 3.1]{RS3}.

The following corollary is a~consequence of Theorem \ref{thm4} and Proposition \ref{prop-cru}.
\begin{cor} \label{thm5}
Let $A$ be an algebra. The following statements are equivalent:
\begin{enumerate}
\renewcommand{\labelenumi}{\rm(\roman{enumi})}
\item $\Gamma_A$ admits a~separating family of components containing of at least two components.
\item $\Gamma_A$ admits a~separating family of components containing infinitely many stable tubes.
\end{enumerate}
\end{cor}

The next theorem is a~consequence of Theorems \ref{thm7} and \ref{thm4}, and describes the supports of indecomposable modules in the module categories
of algebras with separating families of Auslander-Reiten components.
\begin{thm}\label{thm8}%
Let $A$ be an algebra with a~separating family of components in $\Gamma_A$.
Then there exist quotient algebras $B_1, \ldots, B_n$ of $A$ such that the following statements hold.
\begin{enumerate}
\renewcommand{\labelenumi}{\rm(\roman{enumi})}
\item For each $i\in\{1,\ldots,n\}$, $B_i$ is either a~generalized double tilted algebra or a~generalized multicoil algebra.
\item The indecomposable modules in $\ind B_i$, $i\in\{1,\ldots,n\}$, exhaust all modules in $\ind A$.
\end{enumerate}
\end{thm}

In particular, we have the following direct consequence of Theorem \ref{thm8} and results proved in \cite{MS2, RS3}.
\begin{cor}\label{cor9}%
Let $A$ be an algebra with a~separating family of components in $\Gamma_A$. Then for all but finitely many isomorphism classes of modules $M$ in $\ind A$
there exists a~quotient algebra $B$ of $A$ such that $\gldim B\leq 3$ and $M$ is a~module in $\ind B$ with $\pd_BM\leq 2$ and $\id_BM\leq 2$.
\end{cor}

In \cite{CB1, CB2} Crawley-Boevey introduced the concept of a~generically tame algebra. An indecomposable right $A$-module $M$ over an algebra $A$ is called
a~generic module if $M$ is of infinite length over $A$ but of finite length over $\End_A(M)$, called the endolength of $M$. Then an algebra $A$ is called
generically tame if, for any positive integer $d$, there are only finitely many isomorphism classes of generic right $A$-modules of endolength $d$.
An algebra $A$ is called generically finite if there are at most finitely many pairwise non-isomorphic generic right $A$-modules.
Further, $A$ is called generically of polynomial growth if there is a~positive integer $m$ such that for any positive integer $d$ the number of isomorphism
classes of generic right $A$-modules of endolength $d$ is at most $d^m$.
We note that every algebra $A$ of finite representation type is generically trivial, that is, there is no generic right $A$-module.
We also stress that by a~theorem of Crawley-Boevey \cite[Theorem 4.4]{CB1}, if $A$ is an algebra over an algebraically closed field $K$,
then $A$ is generically tame if and only if $A$ is tame in the sense of Drozd \cite{Dr} (see also \cite{CB, SS2}).

Recall also that following \cite{S3.5} the component quiver $\Sigma_A$ of an algebra $A$ has the components of $\Gamma_A$ as vertices and
there is an arrow $\sC\to\sD$ in $\Sigma_A$ if $\rad_A^{\infty}(X,Y)\neq 0$, for some modules $X$ in $\sC$ and $Y$ in $\sD$.
In particular, a~component $\sC$ of $\Gamma_A$ is generalized standard if and only if there is no loop at $\sC$ in $\Sigma_A$.

The next result characterizes the generically tame algebras with separating families of Auslander-Reiten components.
\begin{thm}\label{thm10}%
Let $A$ be an algebra with a~separating family of components in $\Gamma_A$. The following statements are equivalent:
\begin{enumerate}
\renewcommand{\labelenumi}{\rm(\roman{enumi})}
\item $A$ is generically tame.
\item $A$ is generically of polynomial growth.
\item $A^{(l)}$ and $A^{(r)}$ are products of tilted algebras of Euclidean type or tubular algebras.
\item $\Gamma_A$ is almost periodic.
\item $\Sigma_A$ is acyclic.
\end{enumerate}
\end{thm}
\begin{cor}\label{cor11}%
Let $A$ be an algebra with a~separating family of components in $\Gamma_A$. The following statements are equivalent:
\begin{enumerate}
\renewcommand{\labelenumi}{\rm(\roman{enumi})}
\item $A$ is generically finite.
\item $A^{(l)}$ and $A^{(r)}$ are products of tilted algebras of Euclidean type.
\item All but finitely many components of $\Gamma_A$ are stable tubes of rank one.
\end{enumerate}
\end{cor}

The final result provides homological characterizations of generically tame algebras with separating families of Auslander-Reiten components.
\begin{thm}\label{thm12}%
Let $A$ be an algebra with a~separating family of components in $\Gamma_A$. Then the following statements are equivalent:
\begin{enumerate}
\renewcommand{\labelenumi}{\rm(\roman{enumi})}
\item $A$ is generically tame.
\item For all but finitely many isomorphism classes of modules $M$ in $\ind A$ we have $|\Ext_A^1(M,M)| \leq |\End_A(M)|$.
\item For all but finitely many isomorphism classes of modules $M$ in $\ind A$ we have $|\Ext_A^1(M,M)| \leq |\End_A(M)|$ and $\Ext_A^r(M,M)=0$ for any $r\geq 2$.
\item For all but finitely many isomorphism classes of modules $M$ in $\ind A$ the Euler characteristic $\chi_A(M)$ is defined and nonnegative.
\end{enumerate}
\end{thm}

This paper is organized as follows. In Section \ref{sec2} we present and prove several results applied in the proofs of main result of the paper.
Sections \ref{sec3} and \ref{sec4} are devoted to the proofs of Theorems \ref{thm2} and \ref{thm1}, and illustrating examples.
In Sections \ref{sec5} and \ref{sec6} we prove Theorems \ref{thm7} and \ref{thm4}. The final Section \ref{sec7} is devoted to the proofs
of Theorem \ref{thm10}, Corollary \ref{cor11}, and Theorem \ref{thm12}.

For general results on the relevant representation theory we refer to the books \cite{ASS, ARS, Ri1, SS1, SS2, SY3, SY4} and the survey
articles \cite{CB2, MS4, Ri0, Ri-86}.

\section{Preliminary results} \label{sec2}

The aim of this section is to present several key results for the proofs of Theorems \ref{thm1} and \ref{thm2}, and relevant background.

Let $A$ be an algebra and $\sC$ be an infinite almost periodic component of $\Gamma_A$. An indecomposable module $X$ in $\sC$ lying on an oriented cycle of $\sC$
is said to be \emph{cyclic}, and otherwise \emph{acyclic}.
Following \cite{MS1}, we denote by $_c\sC$ the full translation subquiver of $\sC$ obtained by removing all acyclic modules and the arrows attached to them,
and call it the \emph{cyclic part} of $\sC$. The connected translation subquivers of $_c\sC$ are said to be \emph{cyclic components} of $\sC$.
It was shown in \cite[Lemma 5.1]{MS1} that two modules $X$ and $Y$ in $_c\sC$ belong to the same cyclic component of $\sC$ if there is an~oriented cycle in $\sC$
passing through $X$ and $Y$. An indecomposable module $X$ in $\sC$ is said to be \emph{right coherent} if there is in $\sC$ an infinite sectional path
\[X = X_1 \longrightarrow X_2 \longrightarrow \cdots \longrightarrow X_i\longrightarrow X_{i+1} \longrightarrow X_{i+2} \longrightarrow \cdots \]
Dually, an indecomposable module $Y$ in $\sC$ is said to be \emph{left coherent} if there is in $\sC$ an infinite sectional path
\[ \cdots \longrightarrow Y_{j+2} \longrightarrow Y_{j+1} \longrightarrow Y_j \longrightarrow \cdots \longrightarrow Y_2 \longrightarrow Y_1 = Y. \]
A~module $Z$ in $\sC$ is said to be \emph{coherent} if $Z$ is left and right coherent. We denote by $_c\sC^{\coh}$ the full translation subquiver of $_c\sC$
given by all coherent modules in $_c\sC$, and call it the \emph{coherent cyclic part} of $\sC$.
We note that $_c\sC$ may have finite cyclic components, which are obviously not coherent. We will see below that $_c\sC^{\coh}$ is the disjoint union of the
coherent parts of all infinite cyclic components of $\sC$.

In our paper \cite{MS1} we introduced the concept of a~generalized multicoil in order to describe the shape and combinatorial properties of almost cyclic components
with all indecomposable projective modules (right) coherent and all indecomposable injective modules (left) coherent. Namely, a~connected translation quiver
$\Gamma$ is called in \cite{MS1} a~generalized multicoil if $\Gamma$ can be obtained from a~finite family $\cT_1, \cT_2, \ldots, \cT_s$ of stable tubes by an iterated
application of admissible operations (ad~1), (ad~2), (ad~3), (ad~4), (ad~5) and their duals (ad~1$^*$), (ad~2$^*$), (ad~3$^*$), (ad~4$^*$), (ad~5$^*$).
We refer to \cite[Section 2]{MS1} for a~detailed description of these admissible operations and generalized multicoils.
In particular, one knows that all arrows of a~generalized multicoil have trivial valuation.

We have the following consequence of \cite[Theorem A]{MS1}.
\begin{prop}\label{prop-apcoh}%
Let $A$ be an algebra and $\sC$ an infinite almost periodic component of $\Gamma_A$, and $\Gamma$ a~component of $_c\sC^{\coh}$. The following statements hold.
\begin{enumerate}
\renewcommand{\labelenumi}{\rm(\roman{enumi})}
\item $\Gamma$ is the cyclic part of a~generalized multicoil.
\item There is a~finite subquiver $\Sigma_{\Gamma}^{(l)}$ of $\Gamma$ which is a~disjoint union of sectional paths such that every path in $\sC$ from a~module in
$\sC\setminus\Gamma$ to a~module in $\Gamma$ intersects $\Sigma_{\Gamma}^{(l)}$.
\item There is a~finite subquiver $\Sigma_{\Gamma}^{(r)}$ of $\Gamma$ which is a~disjoint union of sectional paths such that every path in $\sC$ from a~module in
$\Gamma$ to a~module in $\sC\setminus\Gamma$ intersects $\Sigma_{\Gamma}^{(r)}$.
\item Every module in $\Gamma$ is a~successor of a~module in $\Sigma_{\Gamma}^{(l)}$ and a~predecessor of a~module in $\Sigma_{\Gamma}^{(r)}$.
\end{enumerate}
\end{prop}
We call $\Sigma_{\Gamma}^{(l)}$ and $\Sigma_{\Gamma}^{(r)}$ the \emph{left border} and the \emph{right border} of $\Gamma$, respectively.
Further, we denote by $\cT_{\Gamma}$ the familly of all indecomposable modules in $\Gamma$ forming the stable tubes used to create the cyclic
generalized multicoil $\Gamma$ by iterated application of admissible operations (ad~1)-(ad~5) and (ad~1$^*$)-(ad~5$^*$), and call it the \emph{tubular part}
of $\Gamma$. Moreover, we denote by $M_{\Gamma}^{(t)}$ the direct sum of all indecomposable modules lying on the mouth of the stable tubes in $\cT_{\Gamma}$,
and call it the \emph{tubular module} of $\Gamma$.
We note that $\rad_A(M_{\Gamma}^{(t)},M_{\Gamma}^{(t)})=0$ if and only if $M_{\Gamma}^{(t)}$ is a~direct sum of pairwise orthogonal bricks.

The following consequence of \cite[Sections 2 and 5]{MS1} is also essential for our considerations.
\begin{prop}\label{prop-order-ad}%
Let $A$ be an algebra, $\sC$ an infinite almost periodic component of $\Gamma_A$, $\Gamma$ a~component of $_c\sC^{\coh}$, and $\Omega$ a~generalized
multicoil enlargement of the family $\cT_{\Gamma}$ of stable tubes such that $\Gamma = _c\Omega$.
Then the following statements hold.
\begin{enumerate}
\renewcommand{\labelenumi}{\rm(\roman{enumi})}
\item $\Omega$ is obtained from $\cT_{\Gamma}$ by iterated application of operations of type (ad~1) followed by operations of types (ad~1$^*$)-(ad~5$^*$).
\item $\Omega$ is obtained from $\cT_{\Gamma}$ by iterated application of operations of type (ad~1$^*$) followed by operations of types (ad~1)-(ad~5).
\end{enumerate}
\end{prop}
We note that an~iterated application of operations of type (ad~1) (respectively, (ad~1$^*$)) to a~stable tube leads to a~ray tube (respectively, a~coray tube)
in the sense of D'Este and Ringel \cite{DER}.

In \cite{MS2} we introduced the concept of a~generalized multicoil enlargement of a~product $C$ of algebras with respect to a~finite family $\cT_1, \cT_2, \ldots, \cT_s$
of pairwise orthogonal generalized standard stable tubes of $\Gamma_C$. Then such a~generalized multicoil enlargement $A$ of $C$ is obtained from
$\cT_1, \cT_2, \ldots, \cT_s$ by iterated application of admissible algebra operations (ad~1), (ad~2), (ad~3), (ad~4), (ad~5) and their duals
(ad~1$^*$), (ad~2$^*$), (ad~3$^*$), (ad~4$^*$), (ad~5$^*$). Then $\Gamma_A$ admits a~generalized standard generalized multicoil $\Gamma$ obtained
from the stable tubes $\cT_1, \cT_2, \ldots, \cT_s$ by iterated application of admissible translation quiver operations corresponding
to the admissible algebra operations leading from $C$ to $A$.

The following proposition follows from \cite[Section 3]{MS2}.
\begin{prop}\label{prop-genst}%
Let $A$ be an algebra, $\sC$ an infinite almost cyclic component of $\Gamma_A$, $\Gamma$ a~component of $_c\sC^{\coh}$, and $B(\Gamma)$ and $B(\cT_{\Gamma})$ the associated
faithful algebras of $\Gamma$ and $\cT_{\Gamma}$. Then the following statements are equivalent:
\begin{enumerate}
\renewcommand{\labelenumi}{\rm(\roman{enumi})}
\item $\Gamma$ is generalized standard.
\item $\cT_{\Gamma}$ is generalized standard.
\item $\cT_{\Gamma}$ is a~finite faithful family of pairwise orthogonal generalized standard stable tubes in $\Gamma_{B(\cT_{\Gamma})}$, and $B(\Gamma)$ is
a~generalized multicoil enlargement of $B(\cT_{\Gamma})$ with respect to $\cT_{\Gamma}$.
\end{enumerate}
\end{prop}

We have also the following consequence of \cite[Sections 3 and 4]{MS2}.
\begin{prop}\label{prop-gme}%
Let $B$ be a~generalized multicoil enlargement of an algebra $C$ (not necessarily indecomposable) with respect to a~faithful family $\cT$ of
pairwise orthogonal generalized standard stable tubes in $\Gamma_{C}$.
Then the following statements hold.
\begin{enumerate}
\renewcommand{\labelenumi}{\rm(\roman{enumi})}
\item $B$ can be obtained from $C$ by iterated application of algebra operations of type (ad~1) followed by algebra operations of types (ad~1$^*$)-(ad~5$^*$).
\item $B$ can be obtained from $C$ by iterated application of algebra operations of type (ad~1$^*$) followed by algebra operations of types (ad~1)-(ad~5).
\end{enumerate}
\end{prop}

We note that an iterated application of algebra operations of type (ad~1) (respectively, (ad~1$^*$)) to the family $\cT$ leads to a~tubular extension
(respectively, tubular coextension) of $C$ in the sense of Ringel \cite{Ri1, SS2}.

The following proposition is relevant and provides a~criterion for a~stable tube to be generalized standard.
\begin{prop}\label{prop-brick}%
Let $A$ be an algebra and $\sT$ be a~stable tube of $\Gamma_A$. The following statements are equivalent:
\begin{enumerate}
\renewcommand{\labelenumi}{\rm(\roman{enumi})}
\item $\sT$ is generalized standard.
\item The mouth modules of $\sT$ are pairwise orthogonal bricks.
\end{enumerate}
\end{prop}
\begin{proof}
(i) $\Longrightarrow$ (ii). Assume $\sT$ is generalized standard, and $M, N$ be two modules lying on the mouth of $\sT$.
We claim that $\rad_A(M,N)=0$.
For each arrow $X \buildrel {\alpha}\over {\hbox to 6mm{\rightarrowfill}} Y$ in $\sT$ we choose an irreducible homomorphism $f_{\alpha}: X\to Y$.
We may assume that $f_{\eta}f_{\xi}\in\rad^3_A$ for any mesh in $\sT$ of the form
$$\xymatrix@C=14pt@R=14pt{
\tau_A Z\ar[rd]_{\xi}&&Z \\
&W\ar[ru]_{\eta} \\
}$$
with $Z$ lying on the mouth of $\sT$ and $f_{\beta}f_{\alpha}+f_{\delta}f_{\gamma}\in\rad^3_A$ for any mesh in $\sT$ of the form
$$\xymatrix@C=14pt@R=14pt{
&U\ar[rd]^{\beta} \\
\tau_A Z\ar[rd]_{\gamma}\ar[ru]^{\alpha}&&Z \\
&V\ar[ru]_{\delta} \\
}$$
Observe that for any irreducible homomorphism $f: X\to Y$ with $X$ and $Y$ from $\sT$, there are automorphisms $b: X\to X$ and $c: Y\to Y$ such that
$$
f_{\alpha}b + \rad^2_A(X,Y)=f + \rad^2_A(X,Y) = cf_{\alpha} + \rad^2_A(X,Y),
$$
where $X \buildrel {\alpha}\over {\hbox to 6mm{\rightarrowfill}} Y$ is the corresponding arrow in $\sT$. This follows from the fact that
$$
\dim_{F(X)}\rad_A(X,Y)/\rad^2_A(X,Y)=1 \,\,{\rm and}\,\, \dim_{F(Y)}\rad_A(X,Y)/\rad^2_A(X,Y)=1,
$$
where $F(X)=\End_A(X)/\rad(\End_A(X))$ and $F(Y)=\End_A(Y)/\rad(\End_A(Y))$.
Let $r$ be the rank of $\sT$. Then $\rad_A(M,N) = \rad_A^{2s}(M,N)$, where $s$ is the smallest positive integer such that $M=\tau_A^sN$.
We note that $s\in\{1, \ldots, r\}$, and $s=r$ if and only if $N=M$. Moreover, any nontrivial path in $\sT$ from $M$ to $N$ is of length $2s+2ri$
for some $i\geq 0$.
This implies that $\rad_A^{2s+2rj+1}(M,N) = \rad_A^{2s+2r(j+1)}(M,N)$ for all $j\geq 0$. We claim that $\rad_A^{t}(M,N)=0$ for all $t\geq 2s$.
It is enough to show that $\rad_A^{t}(M,N)\subseteq\rad_A^{t+1}(M,N)$ for any $t\geq 2s$. Indeed, then $\rad_A^{2s}(M,N)=\rad_A^{\infty}(M,N)=0$,
because $\sT$ is generalized standard, and consequently $\rad_A(M,N)=0$. Consider the mesh
$$\xymatrix@C=14pt@R=14pt{
M\ar[rd]_{\rho}&&\tau_A^{-1}M \\
&E\ar[ru]_{\sigma} \\
}$$
Let $t\geq 2s$ and $\varphi\in\rad_A^{t}(M,N)$. Then we have the equality
\[ \varphi + \rad_A^{t+1}(M,N) = \sum_iv_iu_if_{\sigma}f_{\rho} + \rad_A^{t+1}(M,N), \]
where $u_if_{\sigma}f_{\rho}$ are composites of $t$ irreducible homomorphisms.
Since $f_{\sigma}f_{\rho}\in\rad_A^3$, we get $\varphi + \rad_A^{t+1}(M,N) = 0 + \rad_A^{t+1}(M,N)$, and hence $\varphi\in\rad_A^{t+1}(M,N)$.
This proves our claim. Therefore, the mouth modules of $\sT$ are pairwise orthogonal bricks.

(ii) $\Longrightarrow$ (i). Assume $\sT$ is not generalized standard.
Then there are indecomposable modules $X$ and $Y$ in $\sT$ with $\rad_A^{\infty}(X,Y)\neq 0$. We will prove that then $\rad_A^{\infty}(M,N)\neq 0$
for some modules $M$ and $N$ lying on the mouth of $\sT$. Clearly, there is nothing to show if $\ql(X)=1$ and $\ql(Y)=1$.
Assume that $\ql(Y)\geq 2$.
Then we have in $\sT$ an infinite sectional path
\[ \cdots\to Y_r\to Y_{r-1}\to\cdots\to Y_1\to Y_0 = Y \]
and an arrow $U\to Y$ such that $U\oplus Y_1$ is a~direct summand of the middle term of an almost split sequence in $\mo A$ with the right term
$Y$. Moreover, we have in $\mo A$ an almost split sequence
\[ 0\to U \buildrel {\left[\begin{smallmatrix}f\\g\end{smallmatrix}\right]}\over {\hbox to 10mm{\rightarrowfill}} Y\oplus W\buildrel {[u,v]} \over {\hbox to 10mm{\rightarrowfill}}
V\to 0, \]
where $W=0$ if $\ql(Y)=2$. Take now a~nonzero homomorphism $h$ in $\rad_A^{\infty}(X,Y)$. If $uh\neq 0$, then $uh$ is
a~nonzero homomorphism in $\rad_A^{\infty}(X,V)$ and $\ql(X)+\ql(V)=\ql(X)+\ql(Y)-1$. Assume $uh=0$. Then there is a~homomorphism
$h': X\to U$ such that $h=fh'$. Clearly, $h'\neq 0$. We claim that $h'\in\rad_A^{\infty}(X,U)$. Suppose $h'\not\in\rad_A^{\infty}(X,U)$.
Then there is a~nonnegative integer $s$ such that  $h'\in\rad_A^{s}(X,U)\setminus\rad_A^{s+1}(X,U)$. Then, applying \cite[Corollary 1.6]{L1},
we conclude that $h=fh'\in\rad_A^{s+1}(X,Y)\setminus\rad_A^{s+2}(X,Y)$, a~contradiction with $h\in\rad_A^{\infty}(X,Y)$.
Therefore, $h'$ is a~nonzero homomorphism in $\rad_A^{\infty}(X,U)$ and $\ql(X)+\ql(U)=\ql(X)+\ql(Y)-1$. If $\ql(X)\geq 2$, then applying
dual arguments, we prove that $\rad_A^{\infty}(Z,Y)\neq 0$ for some indecomposable module $Z$ in $\sT$ with $\ql(Z)=\ql(X)-1$.
Summing up, we conclude (by decreasing induction on $\ql(X)+\ql(Y)$) that $\rad_A^{\infty}(M,N)\neq 0$
for some mouth modules $M$ and $N$ in $\sT$, and hence $\rad_A(M,N)\neq 0$. Therefore, (ii) implies (i).
\end{proof}

We note that if $A$ is an~algebra and $\sT$ a~faithful generalized standard stable tube of $\Gamma_A$, then $\pd_AX\leq 1$ and $\id_AX\leq 1$
for any module $X$ in $\sC$ (see \cite[Lemma 5.9]{S2}). But such an algebra $A$ may have an~arbitrary (finite or infinite) global dimension (see \cite{S6}).
We refer also to \cite{S5} for results on the composition factors of modules lying in generalized standard stable tubes.

Let $A$ be an algebra and $\sC$ be an almost periodic component of $\Gamma_A$. Recall that an indecomposable module $X$ in $\sC$ is called \emph{left stable}
(respectively, \emph{right stable}) if $\tau_A^nX$ is nonzero for all $n\geq 0$ (respectively, $n\leq 0$), and \emph{stable} if it is left stable and right stable.
Following \cite{L2}, we denote by $_l\sC$ the \emph{left stable part} of $\sC$, obtained by removing the $\tau_A$-orbits containing projective modules, and
by $_r\sC$ the \emph{right stable part} of $\sC$, obtained by removing the $\tau_A$-orbits containing injective modules. Moreover, we denote by $_l\sC^{\ac}$
(respectively, $_r\sC^{\ac}$) the union of components in $_l\sC$ (respectively, $_r\sC$) consisting entirely of acyclic modules.
We call $_l\sC^{\ac}$ the \emph{left stable acyclic part} of $\sC$ and $_r\sC^{\ac}$ the \emph{right stable acyclic part} of $\sC$. We note that all components
in $_l\sC^{\ac}$ and $_r\sC^{\ac}$ have only finitely many $\tau_A$-orbits, because $\sC$ is assumed to be almost periodic.
Let $\Gamma$ be a~component of $_l\sC^{\ac}$. Then it follows from \cite[Section 3]{L2} that there exists a~finite connected acyclic full valued subquiver
$\Delta_{\Gamma}^{(l)}$ of $\Gamma$ which intersects every $\tau_A$-orbit in $\Gamma$ exactly once, and $\sD_{\Gamma}^{(l)}={\Bbb N}\Delta_{\Gamma}^{(l)}$
is a~full translation subquiver of $\Gamma$ closed under predecessors. We denote by $M_{\Gamma}^{(l)}$ the direct sum of all indecomposable modules lying
on $\Delta_{\Gamma}^{(l)}$.
Dually, let $\Gamma$ be a~component of $_r\sC^{\ac}$. Then, applying \cite[Section 3]{L2}, we conclude that there exists a~finite connected acyclic full valued subquiver
$\Delta_{\Gamma}^{(r)}$ of $\Gamma$ which intersects every $\tau_A$-orbit in $\Gamma$ exactly once, and $\sD_{\Gamma}^{(r)}=({-\Bbb N})\Delta_{\Gamma}^{(r)}$
is a~full translation subquiver of $\Gamma$ closed under successors. We denote by $M_{\Gamma}^{(r)}$ the direct sum of all indecomposable modules lying
on $\Delta_{\Gamma}^{(r)}$.
In case $\Gamma$ is a~stable component, that is a~component of $_l\sC^{\ac}$ and a~component of $_r\sC^{\ac}$, we choose
$\Delta_{\Gamma}^{(l)}$ and $\Delta_{\Gamma}^{(r)}$ such that $\sD_{\Gamma}^{(l)}$ and $\sD_{\Gamma}^{(r)}$ have no common modules.
\begin{prop}\label{prop-D-left}%
Let $A$ be an algebra, $\sC$ an almost periodic component of $\Gamma_A$ with $_l\sC^{\ac}$ nonempty, and $\Gamma$ a~component of $_l\sC^{\ac}$.
Then the following conditions are equivalent:
\begin{enumerate}
\renewcommand{\labelenumi}{\rm(\roman{enumi})}
\item $\sD_{\Gamma}^{(l)}$ is generalized standard.
\item $\Hom_A(M_{\Gamma}^{(l)},\tau_AM_{\Gamma}^{(l)})=0$.
\item $A_{\Gamma}^{(l)}=B(M_{\Gamma}^{(l)})$ is a~tilted algebra of the form $\End_{H_{\Gamma}^{(l)}}(T_{\Gamma}^{(l)})$, for a~hereditary algebra $H_{\Gamma}^{(l)}$
of type $\Delta_{\Gamma}^{(l)}$ and a~tilting module $T_{\Gamma}^{(l)}$ in $\mo H_{\Gamma}^{(l)}$ without indecomposable preinjective direct summands, and
$\sD_{\Gamma}^{(l)}$ is the image of the preinjective component $\cQ_{\Gamma}^{(l)}$ of $\Gamma_{H_{\Gamma}^{(l)}}$ via the functor
$\Hom_{H_{\Gamma}^{(l)}}(T_{\Gamma}^{(l)},-): \mo H_{\Gamma}^{(l)} \to \mo A_{\Gamma}^{(l)}$.
\end{enumerate}
\end{prop}
\begin{proof}
(i)$\Longrightarrow$(iii). Assume that $\sD_{\Gamma}^{(l)}$ is generalized standard. Observe that the faithful algebra $A_{\Gamma}^{(l)}$ of $M_{\Gamma}^{(l)}$
is the faithful algebra $B(\sD_{\Gamma}^{(l)})$ of $\sD_{\Gamma}^{(l)}$. There exists a~module $N$ in the additive category $\add(\sD_{\Gamma}^{(l)})$
such that $\ann_A(\sD_{\Gamma}^{(l)})=\ann_A(N)$. Clearly, we have $\ann_A(\sD_{\Gamma}^{(l)})\subseteq \ann_A(M_{\Gamma}^{(l)})$.
Consider a~monomorphism $u: N\to E_A(N)$ from $N$ to its injective envelope $E_A(N)$ in $\mo A$. Since $\Delta_{\Gamma}^{(l)}$ is connected and intersects
every $\tau_A$-orbit in $\sD_{\Gamma}^{(l)}$, we conclude that $\sD_{\Gamma}^{(l)}\setminus \Delta_{\Gamma}^{(l)}$ does not contain an injective module.
Then $u$ factors through a~module $M$ from $\add(M_{\Gamma}^{(l)})$, and hence there is a~monomorphism $v: N\to M$.
But then $\ann_A(M_{\Gamma}^{(l)})\subseteq\ann_A(M)\subseteq\ann_A(N)=\ann_A(\sD_{\Gamma}^{(l)})$.
Therefore, indeed $A_{\Gamma}^{(l)} = B(\sD_{\Gamma}^{(l)})$. Moreover, because $A_{\Gamma}^{(l)}$ is a~quotient algebra of $A$, $\sD_{\Gamma}^{(l)}$
is a~full translation subquiver of the Auslander-Reiten quiver $\Gamma_{A_{\Gamma}^{(l)}}$ of $A_{\Gamma}^{(l)}$.
Let $\sE_{\Gamma}^{(l)}$ be the component of $\Gamma_{A_{\Gamma}^{(l)}}$ containing $\sD_{\Gamma}^{(l)}$. We claim that $\Delta_{\Gamma}^{(l)}$ is a~section
of $\sE_{\Gamma}^{(l)}$. It is enough to show that $\sE_{\Gamma}^{(l)}$ does not contain an indecomposable projective $A_{\Gamma}^{(l)}$-module.
Suppose that $\sE_{\Gamma}^{(l)}$ contains a~projective module. We may choose minimal $m\in{\Bbb N}$ such that $\tau_{A_{\Gamma}^{(l)}}^{-m}\Delta_{\Gamma}^{(l)}$
contains an indecomposable direct summand, say $R$, of the radical of an indecomposable projective module $P$ in $\sE_{\Gamma}^{(l)}$.
Then considering the projective cover of $R$ in $\mo A_{\Gamma}^{(l)}$ we conclude that there is an epimorphism from a~module in $\add(M_{\Gamma}^{(l)})$ to $R$,
and hence a~nonzero homomorphism $f: X\to R$ with $X$ a~module lying on $\Delta_{\Gamma}^{(l)}$.
Further, since $M_{\Gamma}^{(l)}$ is a~faithful $\Delta_{\Gamma}^{(l)}$-module, there is a~monomorphism from $P$ to a~module $N$ in $\add(M_{\Gamma}^{(l)})$,
and hence $gf\neq 0$ for a~homomorphism $g: R\to Y$ with $Y$ a~module lying on $\Delta_{\Gamma}^{(l)}$. This implies that $gf$ is a~nonzero homomorphism in
$\rad_{A_{\Gamma}^{(l)}}^{\infty}(X,Y)$, because there is no path in $\sE_{\Gamma}^{(l)}$ from $X$ to $Y$ passing through $R$.
Clearly, $\rad_{A_{\Gamma}^{(l)}}^{\infty}(X,Y)\neq 0$ implies $\rad_{A}^{\infty}(X,Y)\neq 0$, a~contradiction because $\sD_{\Gamma}^{(l)}$ is generalized standard.
Therefore, $\Delta_{\Gamma}^{(l)}$ is a~faithful section of the component of $\sE_{\Gamma}^{(l)}$, and
$\Hom_{A_{\Gamma}^{(l)}}(M_{\Gamma}^{(l)},\tau_{A_{\Gamma}^{(l)}}M_{\Gamma}^{(l)})=\Hom_{A}(M_{\Gamma}^{(l)},\tau_{A}M_{\Gamma}^{(l)})=0$.

Applying now the criterion of Liu and Skowro\'nski (see \cite[Theorem VIII.5.6]{ASS} or \cite[Theorem VIII.7.7]{SY4}) we conclude that
$H_{\Gamma}^{(l)} = \End_{A_{\Gamma}^{(l)}}(M_{\Gamma}^{(l)}) = \End_{A}(M_{\Gamma}^{(l)})$ is a~hereditary algebra,
$T_{\Gamma}^{(l)}=D(M_{\Gamma}^{(l)})$ is a~tilting module in $\mo H_{\Gamma}^{(l)}$, there is a~canonical isomorphism of algebras
$A_{\Gamma}^{(l)}\to\End_{H_{\Gamma}^{(l)}}(T_{\Gamma}^{(l)})$, and $\sE_{\Gamma}^{(l)}$ is the connecting component of $\Gamma_{A_{\Gamma}^{(l)}}$ with
$\Delta_{\Gamma}^{(l)}$ the section determined by $T_{\Gamma}^{(l)}$. Moreover, $T_{\Gamma}^{(l)}$ is without an indecomposable preinjective
$H_{\Gamma}^{(l)}$-module, because $\sD_{\Gamma}^{(l)}$ does not contain a~projective module (see \cite[Proposition VIII.6.9]{SY4}).
Finally, the translation quiver $\sD_{\Gamma}^{(l)}$ is the image of the preinjective component $\cQ_{\Gamma}^{(l)}$ of $\Gamma_{H_{\Gamma}^{(l)}}$ via the functor
$\Hom_{H_{\Gamma}^{(l)}}(T_{\Gamma}^{(l)},-): \mo H_{\Gamma}^{(l)} \to \mo A_{\Gamma}^{(l)}$ (see \cite[Theorem VIII.6.7]{SY4}).

(iii)$\Longrightarrow$(ii). Assume that (iii) holds. Then the functor
$\Hom_{H_{\Gamma}^{(l)}}(T_{\Gamma}^{(l)},-): \mo H_{\Gamma}^{(l)} \to \mo A_{\Gamma}^{(l)}$ induces an equivalence of categories
$\add(\cQ_{\Gamma}^{(l)})\to \add(\sD_{\Gamma}^{(l)})$, with $\cQ_{\Gamma}^{(l)}$ the preinjective component of $\Gamma_{H_{\Gamma}^{(l)}}$.
Hence $\Hom_{A_{\Gamma}^{(l)}}(M_{\Gamma}^{(l)},\tau_{A_{\Gamma}^{(l)}}M_{\Gamma}^{(l)})=0$ follows.

(ii)$\Longrightarrow$(i). Assume that (ii) holds.
Suppose that there exist indecomposable modules $X$ and $Y$ in $\sD_{\Gamma}^{(l)}$ such that $\rad_A^{\infty}(X,Y)\neq 0$. Then it follows from
\cite[Lemma 2.1(i)]{S7} that there exist an infinite path
\[  X = X_0 \buildrel {f_1}\over {\hbox to 7mm{\rightarrowfill}} X_1 \buildrel {f_2}\over {\hbox to 7mm{\rightarrowfill}} \cdots \buildrel {f_{r-1}}\over {\hbox to 7mm{\rightarrowfill}}
X_{r-1} \buildrel {f_r}\over {\hbox to 7mm{\rightarrowfill}} X_r \buildrel {f_{r+1}}\over {\hbox to 7mm{\rightarrowfill}} \cdots \]
of irreducible homomorphisms between indecomposable modules in $\mo A$ and $g_r\in\rad_A^{\infty}(X_r,Y)$, $r\geq 1$, such that $g_rf_r\ldots f_1\neq 0$
for any $r\geq 1$. Since $\sD_{\Gamma}^{(l)}$ is an acyclic quiver of the form ${\Bbb N}\Delta_{\Gamma}^{(l)}$ and $X$ belongs to $\sD_{\Gamma}^{(l)}$,
there is $s\geq 1$ such that $X_s$ belongs to $\Delta_{\Gamma}^{(l)}$.
Because $\rad_A^{\infty}(X_s,Y)\neq 0$, applying \cite[Lemma 2.1(ii)]{S7}, we conclude that there exist an infinite path
\[  \cdots \buildrel {h_{t+1}}\over {\hbox to 7mm{\rightarrowfill}} Y_t \buildrel {h_t}\over {\hbox to 7mm{\rightarrowfill}} Y_{t-1} \buildrel {h_{t-1}}\over {\hbox to 7mm{\rightarrowfill}}
\cdots \buildrel {h_{2}}\over {\hbox to 7mm{\rightarrowfill}} Y_{1} \buildrel {h_1}\over {\hbox to 7mm{\rightarrowfill}} Y_0 = Y \]
of irreducible homomorphisms between indecomposable modules in $\mo A$ and $u_t\in\rad_A^{\infty}(X_s,Y_t)$, $t\geq 1$, such that $h_1\ldots h_tu_t\neq 0$
for any $t\geq 1$. Then there exists $m\geq 1$ such that $Y_t$ belongs to the subquiver ${\Bbb N}\tau_A\Delta_{\Gamma}^{(l)}$ of
$\sD_{\Gamma}^{(l)} = {\Bbb N}\Delta_{\Gamma}^{(l)}$, for any $t\geq m$.
Let $U=X_s$, $V=Y_m$, and $f:U\to V$ a~nonzero homomorphism in $\rad_A^{\infty}(U,V)$.
Consider now a~monomorphism $v: V\to E_A(V)$, where $E_A(V)$ is an injective envelope of $V$ in $\mo A$. Observe that the subquiver
${\Bbb N}\tau_A\Delta_{\Gamma}^{(l)}$ has no injective module. Hence $v$ factors through a~module $N$ in $\add(\tau_A\Delta_{\Gamma}^{(l)})$,
and then there is a~monomorphism $w: V\to N$. Then $wu_m\neq 0$, and hence there is an indecomposable direct summand $W$ of $N$ such that $\Hom_A(U,W)\neq 0$.
But then $\Hom_A(M_{\Gamma}^{(l)},\tau_AM_{\Gamma}^{(l)})\neq 0$, because $U$ is a~direct summand of $M_{\Gamma}^{(l)}$ and $W$ is a~direct summand of
$\tau_AM_{\Gamma}^{(l)}$. This shows that (ii) implies (i).
\end{proof}

We have the dual proposition.
\begin{prop}\label{prop-D-left-dual}%
Let $A$ be an algebra, $\sC$ an almost periodic component of $\,\Gamma_A$ with $_r\sC^{\ac}$ nonempty, and $\Gamma$ a~component of $_r\sC^{\ac}$.
Then the following conditions are equivalent:
\begin{enumerate}
\renewcommand{\labelenumi}{\rm(\roman{enumi})}
\item $\sD_{\Gamma}^{(r)}$ is generalized standard.
\item $\Hom_A(\tau_A^{-1}M_{\Gamma}^{(r)},M_{\Gamma}^{(r)})=0$.
\item $A_{\Gamma}^{(r)}=B(M_{\Gamma}^{(r)})$ is a~tilted algebra of the form $\End_{H_{\Gamma}^{(r)}}(T_{\Gamma}^{(r)})$, for a~hereditary algebra $H_{\Gamma}^{(r)}$
of type $\Delta_{\Gamma}^{(r)}$ and a~tilting module $T_{\Gamma}^{(r)}$ in $\mo H_{\Gamma}^{(r)}$ without indecomposable postprojective direct summands, and
$\sD_{\Gamma}^{(r)}$ is the image of the postprojective component $\cP_{\Gamma}^{(r)}$ of $\Gamma_{H_{\Gamma}^{(r)}}$ via the functor
$\Ext_{H_{\Gamma}^{(r)}}^1(T_{\Gamma}^{(r)},-): \mo H_{\Gamma}^{(r)} \to \mo A_{\Gamma}^{(r)}$.
\end{enumerate}
\end{prop}

\section{Proof of Theorem \ref{thm2}} \label{sec3}

The aim of this section is to prove Theorem \ref{thm2}, using the results presented in the previous section.

Let $A$ be an algebra and $\sC$ be an infinite generalized standard component of $\Gamma_A$. Then, by \cite[Theorem 2.3]{S2}, $\sC$ is almost periodic.
We use the notation introduce above. We consider the following ideals of $A$:
\begin{itemize}
\item $I_{\sC}^{(l)}$ the annihilator of the disjoint union of the translation quivers $\sD_{\Gamma}^{(l)}$, for all components $\Gamma$ of $_l\sC^{\ac}$,
\item $I_{\sC}^{(r)}$ the annihilator of the disjoint union of the translation quivers $\sD_{\Gamma}^{(r)}$, for all components $\Gamma$ of $_r\sC^{\ac}$,
\item $I_{\sC}^{(c)}$ the annihilator of the disjoint union of all components $\Gamma$ of $_c\sC^{\coh}$.
\end{itemize}
Further, we consider the quotient algebras $A_{\sC}^{(lt)}=A/I_{\sC}^{(l)}$, $A_{\sC}^{(rt)}=A/I_{\sC}^{(r)}$, $A_{\sC}^{(c)}=A/I_{\sC}^{(c)}$.
Then
it follows from Propositions
\ref{prop-apcoh}, \ref{prop-genst}, \ref{prop-D-left} and \ref{prop-D-left-dual} that the required statements (i), (ii), (iii) are satisfied.

We prove that the statement (iv) holds. Indeed, let $X$ be an acyclic module in $\sC$ which does not belong to a~subquiver of the form $\sD_{\Gamma}^{(l)}$
or $\sD_{\Gamma}^{(r)}$. Then there exists a~nonnegative integer $m_X$ such that $\tau_A^{m_X}X$ is projective, or belongs to $\Delta_{\Gamma}^{(l)}$,
for a~component $\Gamma$ of $_l\sC^{\ac}$, or belongs to the right border $\Sigma_{\Gamma}^{(r)}$ of a~component $\Gamma$ of $_c\sC^{\ac}$.
Similarly, there is a~nonnegative integer $n_X$ such that $\tau_A^{-n_X}X$ is injective, or belongs to $\Delta_{\Gamma}^{(r)}$,
for a~component $\Gamma$ of $_r\sC^{\ac}$, or belongs to the left border $\Sigma_{\Gamma}^{(l)}$ of a~component $\Gamma$ of $_c\sC^{\ac}$.
Hence there are at most finitely many such acyclic modules in $\sC$, and (iv) follows.

The statement (v) will follow from the facts below.
\begin{enumerate}
\renewcommand{\labelenumi}{\rm(\arabic{enumi})}
\item It follows from \cite[Proposition 2.4 and Theorem 2.5]{RS3} that every finite component of $_c\sC$ contains an indecomposable projective module and indecomposable injective module.
Hence there are at most finitely many indecomposable modules lying in finite cyclic components of $_c\sC$.
\item An indecomposable module $X$ in $\sC$ is said to be stable if it is left stable and right stable. We denote by $_s\sC$ the stable part of $\sC$
obtained by removing all nonstable modules and the arrows attached to them. It is known (see \cite{L2}, \cite{Z} or \cite[Theorem IX.4.9]{SY4}) that
an infinite component $\Gamma$ of $_s\sC$ contains an oriented cycle if and only if $\Gamma$ is a~stable tube.
Moreover, every stable tube of $_s\sC$ consists of periodic coherent modules, and hence is contained in $_c\sC^{\coh}$. On the other hand, every finite
component $\Gamma$ of $_s\sC$ containing an oriented cycle consists of periodic modules, contains an immediate predecessor of a~projective module
in $\sC$ (respectively, an immediate successor of an~injective module in $\sC$). Hence the number of indecomposable modules in $_c\sC\setminus{_c\sC^{\coh}}$
lying in $_s\sC$ is finite.
\item Let $X$ be a~module in $_c\sC\setminus{_c\sC^{\coh}}$ such that $\tau_A^nX$ belongs to ${_c\sC^{\coh}}$ for some positive integer $n$.
Then there exist $m\in\{1,\ldots,n\}$ and a~component $\Gamma$ of ${_c\sC^{\coh}}$ such that $\tau_A^mX$ belongs to the right border
$\Sigma_{\Gamma}^{(r)}$ of $\Gamma$.
\item Let $X$ be a~module in $_c\sC\setminus{_c\sC^{\coh}}$ such that $\tau_A^{-n}X$ belongs to ${_c\sC^{\coh}}$ for some positive integer $n$.
Then there exist $m\in\{1,\ldots,n\}$ and a~component $\Gamma$ of ${_c\sC^{\coh}}$ such that $\tau_A^{-m}X$ belongs to the left border
$\Sigma_{\Gamma}^{(l)}$ of $\Gamma$.
\item Let $X$ be a~nonperiodic, nonprojective module in $_c\sC\setminus{_c\sC^{\coh}}$ such that, for any positive integer $r$, $\tau_A^rX$
does not belong to ${_c\sC^{\coh}}$. Then there exists a~positive integer $m$ such that $\tau_A^mX$ is projective or belongs to $\Delta_{\Gamma}^{(l)}$
for a~component $\Gamma$ of ${_l\sC^{\ac}}$.
\item Let $X$ be a~nonperiodic, noninjective module in $_c\sC\setminus{_c\sC^{\coh}}$ such that, for any positive integer $s$, $\tau_A^{-s}X$
does not belong to ${_c\sC^{\coh}}$. Then there exists a~positive integer $n$ such that $\tau_A^{-n}X$ is injective or belongs to $\Delta_{\Gamma}^{(r)}$
for a~component $\Gamma$ of ${_r\sC^{\ac}}$.
\end{enumerate}
Summing up, we conclude that $_c\sC\setminus{_c\sC^{\coh}}$ is finite, which is the statement (v).

\smallskip

We end this section with an example illustrating the above considerations.

\begin{ex} \label{ex-3}
Let $K$ be an algebraically closed field, $Q$ the quiver of the form
%
\[
\xymatrix@C=13pt@R=12pt{
&&\scr{30} \\
\scr{32}&&&\scr{27}\ar[lu]_{\vartheta}\ar[ld]_{\zeta_2} \\
\scr{33}&&\scr{29}\ar[ld]_{\kappa}&&\scr{26}\ar[lu]_{\zeta_1}\ar[ld]^{\zeta_3}&\scr{18}\ar[l]_{n} \\
&\scr{31}\ar[lu]^{\delta_2}\ar[luu]_{\delta_1}\ar[ld]_{\delta_3}\ar[ldd]^{\delta_4}&&\scr{28}\ar[lu]^{\zeta_4}&&&\scr{17}\ar[lu]_{m} \\
\scr{34}&&&\scr{9}\ar[rd]^{\eta}&&\scr{16}\ar[ld]_{\psi}\ar[ru]^{l}&&\scr{19}\ar[rd]^{i}\ar[ll]_{j}\\
\scr{35}&&\scr{10}\ar[ru]^{\xi}\ar[rd]_{\mu}\ar[dd]_{\pi}&&\scr{7}\ar[dddd]^{\rho}&&&&\scr{20}\ar[d]^{h}\\
&&&\scr{8}\ar[ru]_{\nu}&&&&\scr{25}\ar@<2pt>[dd]^{t}\ar@<-2pt>[dd]_{s}&\scr{21}\ar[dd]^{g}\ar[rd]^{f}\\
\scr{0}&\scr{1}\ar[l]_{\theta}&\scr{2}\ar[l]_{\omega}\ar[dd]_{\lambda}&&&&&&&\scr{22}\ar[ld]^{e}\\
&&&\scr{5}\ar[ld]_{\beta}&&\scr{23}\ar[rd]_{r}&&\scr{24}\ar[ll]_{q}\ar[ld]^{p}&\scr{15}\ar[d]^{d}\\
&&\scr{3}&&\scr{6}\ar[ld]^{\sigma}\ar[lu]_{\alpha}\ar[rd]_{\varphi}&&\scr{12}\ar[ld]^{a}\ar[rd]_{b}&&\scr{14}\ar[ld]^{c}\\
&&&\scr{4}\ar[lu]^{\gamma}&&\scr{11}&&\scr{13}\\
}
\]
and $I$ the ideal in the path algebra $KQ$ of $Q$ over $K$ generated by the paths
$\alpha\beta - \sigma\gamma$, $\xi\eta - \mu\nu$,
$\pi\lambda - \xi\eta\rho\alpha\beta$, $\rho\varphi$, $\psi\rho$, $jl$, $dc$, $ed$, $gd$, $hg$, $hf$, $ih$,
$pa$, $pb$, $ra$, $rb$, $qr$, $sq$, $sp$, $tq$, $tp$,
$mn$, $n\zeta_1$, $n\zeta_3$, $\zeta_1\vartheta$, $\zeta_1\zeta_2$, $\zeta_3\zeta_4$, $\zeta_2\kappa$, $\zeta_4\kappa$, $\kappa\delta_1$, $\kappa\delta_2$, $\kappa\delta_3$, $\kappa\delta_4$,
and $A=KQ/I$.
We mention that this is an enlarged version of \cite[Example 6.4]{MPS1}.
The Auslander-Reiten quiver $\Gamma_A$ admits a~generalized standard component $\sC$ obtained by identification of the common simple modules $S_{12}$ and $S_{18}$ occurring
in the following three translation quivers: $\Gamma_1$ of the form
\[ \hspace{-12mm}\xymatrix@C=5.5pt@R=6pt{
&&&&&&&&&&&&&&&& \scr{P_{22}}\ar[rd] && \scr{I_{22}}\ar[rd] \cr
&&&&&&&&&&&&\scr{P_{15}}\ar@{=}[r]& \scr{I_{14}}\ar[rr] && \scr{S_{15}}\ar[rd]\ar[ru] && \scr{R}\ar[rd]\ar[ru] && \scr{S_{21}}\ar[rr] && \scr{P_{20}}\ar[lldddd]\ar@{=}[r]&\scr{I_{21}} \cr
&&&&&&&&&&&&&&&& \scr{P_{21}}\ar[rd]\ar[ru] && \scr{I_{15}}\ar[rd]\ar[ru] \cr
&&&&&&&&&&&&&&& \scr{S_{22}}\ar[ru] && \scr{L}\ar[ru] && \scr{S_{22}} \cr
&&&&&&&&&&&&&&&&&&&&&&& \cr
&&&&&&&&&&&&&&&&&&& \scr{S_{20}}\ar[rd] && \scr{\phantom{|}\circ\phantom{|}}\ar[rd] && \scr{\phantom{|}\circ\phantom{|}}\ar@{--}[ddddddddd] \cr
\scr{\phantom{|}\circ\phantom{|}}\ar[rd]\ar@{--}[dddddddd] && \scr{\phantom{|}\circ\phantom{|}}\ar[rd] && \scr{\phantom{|}\circ\phantom{|}}\ar[rd] && \scr{S_{2}}\ar[rd] && \scr{\phantom{|}\circ\phantom{|}}\ar[rd] && \scr{\phantom{|}\circ\phantom{|}}\ar[rd] && \scr{T}\ar[rd] && \scr{S_{14}}\ar[luuuuu] &&&&&& \scr{P_{19}}\ar[rd]\ar[ru] && \scr{\phantom{|}\circ\phantom{|}}\ar[rd]\ar[ru] \cr
& \scr{\phantom{|}\circ\phantom{|}}\ar[rd]\ar[ru]\ar[r] & \scr{P_{10}}\ar[r] & \scr{\phantom{|}\circ\phantom{|}}\ar[rd]\ar[ru]\ar[r] & \scr{\phantom{|}\circ\phantom{|}}\ar[r] & \scr{\phantom{|}\circ\phantom{|}}\ar[rd]\ar[ru]\ar[r] & \scr{I_{3}}\ar[r] & \scr{\phantom{|}\circ\phantom{|}}\ar[rd]\ar[ru] && \scr{\phantom{|}\circ\phantom{|}}\ar[rd]\ar[ru] && \scr{\phantom{|}\circ\phantom{|}}\ar[rd]\ar[ru]
&& \scr{I_{13}}\ar[rd]\ar[ru] && \scr{S_{18}}\ar[rd] && \scr{S_{17}}\ar[rd] && \scr{\phantom{|}\circ\phantom{|}}\ar[rd]\ar[ru] && \scr{\phantom{|}\circ\phantom{|}}\ar[rd]\ar[ru] && \scr{\phantom{|}\circ\phantom{|}} \cr
\scr{\phantom{|}\circ\phantom{|}}\ar[rd]\ar[ru] && \scr{\phantom{|}\circ\phantom{|}}\ar[rd]\ar[ru] && \scr{\phantom{|}\circ\phantom{|}}\ar[rd]\ar[ru] && \scr{\phantom{|}\circ\phantom{|}}\ar[rd]\ar[ru] && \scr{\phantom{|}\circ\phantom{|}}\ar[rd]\ar[ru] && \scr{\phantom{|}\circ\phantom{|}}\ar[rd]\ar[ru] && \scr{\phantom{|}\circ\phantom{|}}\ar[rd]\ar[ru] && \scr{S_{12}}
&& \scr{P_{17}}\ar[rd]\ar[ru] && \scr{\phantom{|}\circ\phantom{|}}\ar[rd]\ar[ru] && \scr{\phantom{|}\circ\phantom{|}}\ar[rd]\ar[ru] && \scr{\phantom{|}\circ\phantom{|}}\ar[rd]\ar[ru] \cr
& \scr{\phantom{|}\circ\phantom{|}}\ar[rd]\ar[ru] && \scr{\phantom{|}\circ\phantom{|}}\ar[rd]\ar[ru] && \scr{\phantom{|}\circ\phantom{|}}\ar[rd]\ar[ru] && \scr{\phantom{|}\circ\phantom{|}}\ar[rd]\ar[ru] && \scr{\phantom{|}\circ\phantom{|}}\ar[rd]\ar[ru] && \scr{\phantom{|}\circ\phantom{|}}\ar[rd]\ar[ru]
&& \scr{I_{11}}\ar[rd]\ar[ru] &&&& \scr{P_{16}}\ar[rd]\ar[ru] && \scr{\phantom{|}\circ\phantom{|}}\ar[rd]\ar[ru] && \scr{\phantom{|}\circ\phantom{|}}\ar[rd]\ar[ru] && \scr{\phantom{|}\circ\phantom{|}} \cr
\scr{\phantom{|}\circ\phantom{|}}\ar[rd]\ar[ru] && \scr{\phantom{|}\circ\phantom{|}}\ar[rd]\ar[ru] && \scr{E}\ar[rd]\ar[ru] && \scr{\phantom{|}\circ\phantom{|}}\ar[rd]\ar[ru] && \scr{\phantom{|}\circ\phantom{|}}\ar[rd]\ar[ru] && \scr{\phantom{|}\circ\phantom{|}}\ar[rd]\ar[ru] && \scr{\phantom{|}\circ\phantom{|}}\ar[rd]\ar[ru] && \scr{S_{6}}\ar[rd]
&& \scr{S_{7}}\ar[rd]\ar[ru] && \scr{\phantom{|}\circ\phantom{|}}\ar[rd]\ar[ru] && \scr{\phantom{|}\circ\phantom{|}}\ar[rd]\ar[ru] && \scr{\phantom{|}\circ\phantom{|}}\ar[rd]\ar[ru] \cr
& \scr{\phantom{|}\circ\phantom{|}}\ar[rd]\ar[ru] && \scr{\phantom{|}\circ\phantom{|}}\ar[rd]\ar[ru] && \scr{\phantom{|}\circ\phantom{|}}\ar[rd]\ar[ru] && \scr{\phantom{|}\circ\phantom{|}}\ar[rd]\ar[ru] && \scr{\phantom{|}\circ\phantom{|}}\ar[rd]\ar[ru] && \scr{\phantom{|}\circ\phantom{|}}\ar[rd]\ar[ru]
&& \scr{\phantom{|}\circ\phantom{|}}\ar[rd]\ar[ru] && \scr{\phantom{|}\circ\phantom{|}}\ar[rd]\ar[ru] && \scr{\phantom{|}\circ\phantom{|}}\ar[rd]\ar[ru] && \scr{\phantom{|}\circ\phantom{|}}\ar[rd]\ar[ru] && \scr{\phantom{|}\circ\phantom{|}}\ar[rd]\ar[ru] && \scr{\phantom{|}\circ\phantom{|}} \cr
\scr{\phantom{|}\circ\phantom{|}}\ar[rd]\ar[ru] && \scr{\phantom{|}\circ\phantom{|}}\ar[rd]\ar[ru] && \scr{\phantom{|}\circ\phantom{|}}\ar[rd]\ar[ru] && \scr{\phantom{|}\circ\phantom{|}}\ar[rd]\ar[ru] && \scr{\phantom{|}\circ\phantom{|}}\ar[rd]\ar[ru] && \scr{\phantom{|}\circ\phantom{|}}\ar[rd]\ar[ru] && \scr{\phantom{|}\circ\phantom{|}}\ar[rd]\ar[ru] && \scr{\phantom{|}\circ\phantom{|}}\ar[rd]\ar[ru]
&& \scr{\phantom{|}\circ\phantom{|}}\ar[rd]\ar[ru] && \scr{\phantom{|}\circ\phantom{|}}\ar[rd]\ar[ru] && \scr{\phantom{|}\circ\phantom{|}}\ar[rd]\ar[ru] && \scr{\phantom{|}\circ\phantom{|}}\ar[rd]\ar[ru] \cr
& {\phantom{\scr{\phantom{|}\circ\phantom{|}}}}\ar[ru] && {\phantom{\scr{\phantom{|}\circ\phantom{|}}}}\ar[ru] && {\phantom{\scr{\phantom{|}\circ\phantom{|}}}}\ar[ru] && {\phantom{\scr{\phantom{|}\circ\phantom{|}}}}\ar[ru] && {\phantom{\scr{\phantom{|}\circ\phantom{|}}}}\ar[ru] && {\phantom{\scr{\phantom{|}\circ\phantom{|}}}}\ar[ru]
&& {\phantom{\scr{\phantom{|}\circ\phantom{|}}}}\ar[ru] && {\phantom{\scr{\phantom{|}\circ\phantom{|}}}}\ar[ru] && {\phantom{\scr{\phantom{|}\circ\phantom{|}}}}\ar[ru] && {\phantom{\scr{\phantom{|}\circ\phantom{|}}}}\ar[ru] && {\phantom{\scr{\phantom{|}\circ\phantom{|}}}}\ar[ru] && {\phantom{\scr{\phantom{|}\circ\phantom{|}}}} \cr
& {\phantom{\scr{\phantom{|}\circ\phantom{|}}}}\ar@{.}[u] && {\phantom{\scr{\phantom{|}\circ\phantom{|}}}}\ar@{.}[u] && {\phantom{\scr{\phantom{|}\circ\phantom{|}}}}\ar@{.}[u] && {\phantom{\scr{\phantom{|}\circ\phantom{|}}}}\ar@{.}[u] && {\phantom{\scr{\phantom{|}\circ\phantom{|}}}}\ar@{.}[u] && {\phantom{\scr{\phantom{|}\circ\phantom{|}}}}\ar@{.}[u]
&& {\phantom{\scr{\phantom{|}\circ\phantom{|}}}}\ar@{.}[u] && {\phantom{\scr{\phantom{|}\circ\phantom{|}}}}\ar@{.}[u] && {\phantom{\scr{\phantom{|}\circ\phantom{|}}}}\ar@{.}[u] && {\phantom{\scr{\phantom{|}\circ\phantom{|}}}}\ar@{.}[u] && {\phantom{\scr{\phantom{|}\circ\phantom{|}}}}\ar@{.}[u] && {\phantom{\scr{\phantom{|}\circ\phantom{|}}}} \cr
}\]
$\Gamma_2$ of the form
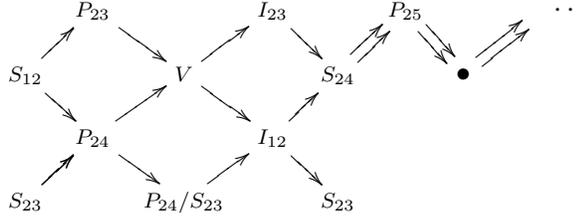
\[ \xymatrix@C=6pt@R=2pt{
&&&&&\scr{P_{23}}\ar[rdd]&&\scr{I_{23}}\ar[rdd]&&\scr{P_{25}}\ar@<2pt>[rdd]\ar@<-2pt>[rdd]&&{\phantom{PP}}\cdots\\
&&&\\
&&&&\scr{S_{12}}\ar[ruu]\ar[rdd]&&\scr{V}\ar[ruu]\ar[rdd]&&\scr{S_{24}}\ar@<2pt>[ruu]\ar@<-2pt>[ruu]&&\bullet\ar@<2pt>[ruu]\ar@<-2pt>[ruu]\\
&&&\\
&&&&&\scr{P_{24}}\ar[ruu]\ar[rdd]&&\scr{I_{12}}\ar[rdd]\ar[ruu] \\
&&& \\
&&&&\scr{S_{23}}\ar[ruu]\ar[ruu]&&\scr{P_{24}/S_{23}}\ar[ruu]&&\scr{S_{23}} \\
}\]
and $\Gamma_3$ of the form
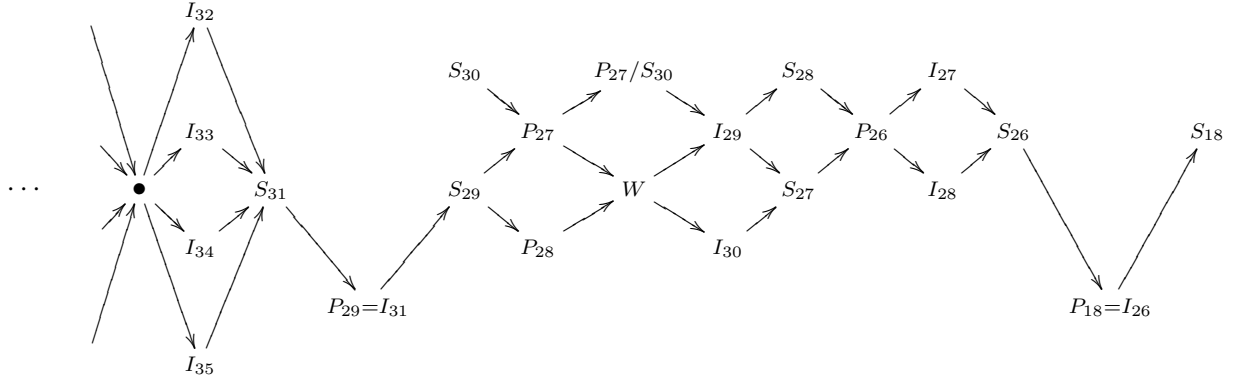
\[ \xymatrix@C=8pt@R=8pt{
&{\phantom{I}}\ar[rddd]&&\scr{I_{32}}\ar[rddd]\\
&&&&&&\scr{S_{30}}\ar[rd]&&\scr{P_{27}/S_{30}}\ar[rd]&&\scr{S_{28}}\ar[rd]&&\scr{I_{27}}\ar[rd] \\
&{\phantom{I}}\ar[rd]&&\scr{I_{33}}\ar[rd]&&&&\scr{P_{27}}\ar[rd]\ar[ru]&&\scr{I_{29}}\ar[rd]\ar[ru]&&\scr{P_{26}}\ar[rd]\ar[ru]&&\scr{S_{26}}\ar[rddd]&&\scr{S_{18}} \\
\cdots&&\bullet\ar[ruuu]\ar[ru]\ar[rd]\ar[rddd]&&\scr{S_{31}}\ar[rdd]&&\scr{S_{29}}\ar[ru]\ar[rd]&&\scr{W}\ar[rd]\ar[ru]&&\scr{S_{27}}\ar[ru]&&\scr{I_{28}}\ar[ru] \\
&{\phantom{I}}\ar[ru]&&\scr{I_{34}}\ar[ru]&&&&\scr{P_{28}}\ar[ru]&&\scr{I_{30}}\ar[ru]\\
&&&&&\scr{P_{29}=I_{31}}\ar[ruu]&&&&&&&&&\scr{P_{18}=I_{26}}\ar[ruuu] \\
&{\phantom{I}}\ar[ruuu]&&\scr{I_{35}}\ar[ruuu] \\
&&&&&&&& \\
}\]
The cyclic part ${}_c{\sC}$ of $\sC$ consists of one infinite component $\Gamma$ and one finite component $\Gamma'$ described as follows.
The infinite cyclic component $\Gamma$ is obtained by removing from $\sC$ the modules $S_{12}, S_{17}, S_{18}, P_{17}$ together with the arrows attached to them,
and the translation quivers $\Gamma_2$, $\Gamma_3$.
The finite cyclic component $\Gamma'$ is the full translation subquiver of $\sC$ given by the vertices $S_{23}, P_{24}, P_{24}/S_{23}, V, I_{12}$.
The maximal cyclic coherent part $_c\Gamma^{\coh}$ of $\Gamma$ is the full translation subquiver of $\sC$ obtained by
removing from $\sC$ the modules $S_{12}$, $I_{13}$, $T$, $S_{14}$, $P_{15}=I_{14}$, $S_{15}$, $P_{21}$, $S_{22}$, $L$, $P_{22}$, $R$, $I_{15}$, $I_{22}$,
$S_{21}$, $P_{20}=I_{21}$, $S_{20}$, $S_{17}$, $P_{17}$, $S_{18}$
together with the arrows attached to them (see \cite[Example 6.4]{MPS1}), and the translation quivers $\Gamma_2$, $\Gamma_3$.
We observe that $S_{22}$ is a~periodic module of $\Gamma$ which does not belong to $_c\Gamma^{\coh}$.

Moreover, we have
\[ P_{\sC} = P_{10}\oplus P_{15}\oplus P_{16}\oplus P_{17}\oplus P_{18}\oplus P_{19}\oplus P_{20}\oplus P_{21}\oplus P_{22}\oplus P_{23}\oplus
P_{24}\oplus P_{25}\oplus P_{26}\oplus P_{27}\oplus P_{28}\oplus P_{29}\oplus P_{30}, \]
\[ P_{\sC}^{\coh} = P_{10}\oplus P_{16}\oplus P_{17}\oplus P_{19}\oplus P_{20},  \]
\[ Q_{\sC} = I_{3}\oplus  I_{11}\oplus  I_{12}\oplus  I_{13}\oplus  I_{14}\oplus  I_{15}\oplus  I_{21}\oplus  I_{22}\oplus  I_{23}\oplus
I_{26}\oplus  I_{27}\oplus  I_{28}\oplus  I_{29}\oplus  I_{30}\oplus  I_{31}\oplus  I_{32}\oplus  I_{33}\oplus  I_{34}\oplus  I_{35},  \]
\[ Q_{\sC}^{\coh} = I_{3}\oplus  I_{11}\oplus  I_{13}\oplus  I_{14},  \]
\[ M_{\sC}^{(l)} = I_{32}\oplus  I_{33}\oplus  I_{34}\oplus  I_{35}\oplus S_{31}, \,\,\,
 M_{\sC}^{(r)} = P_{25}\oplus S_{24}, \,\,\,
 M_{\sC}^{(t)} = S_{6}\oplus S_{7}\oplus E. \]
 We note that $P_{20}$ is a~right coherent projective module and $I_{14}$ is a~left coherent injective module in $\Gamma$ which do not belong
 to $_c\Gamma^{\coh}$.
\end{ex}

\section{Proof of Theorem \ref{thm1}} \label{sec4}

Let $A$ be an algebra and $\sC$ be an infinite component of $\Gamma_A$. Since every generalized standard component is almost periodic,
we may assume that $\sC$ is almost periodic. We use the notation introduced in Section \ref{sec2}. Hence
\begin{itemize}
\item $M_{\sC}^{(l)}$ is the direct sum of all modules $M_{\Gamma}^{(l)}$ given by the quivers $\Delta_{\Gamma}^{(l)}$ associated to all
components $\Gamma$ of the left stable acyclic part $_l\sC^{\ac}$ of $\sC$,
\item $M_{\sC}^{(r)}$ is the direct sum of all modules $M_{\Gamma}^{(r)}$ given by the quivers $\Delta_{\Gamma}^{(r)}$ associated to all
components $\Gamma$ of the right stable acyclic part $_r\sC^{\ac}$ of $\sC$,
\item $M_{\sC}^{(t)}$ is the direct sum of all modules $M_{\Gamma}^{(t)}$ given by the mouth of the tubular parts $\cT_{\Gamma}$ of all
components $\Gamma$ of the cyclic coherent part $_c\sC^{\coh}$ of $\sC$.
\end{itemize}
Further, let
\begin{itemize}
\item $P_{\sC}$ be the direct sum of all projective modules in $\sC$, and $P_{\sC}^{\coh}$ the direct sum of all right coherent projective modules in $\sC$,
\item $Q_{\sC}$ be the direct sum of all injective modules in $\sC$, and $Q_{\sC}^{\coh}$ the direct sum of all left coherent injective modules in $\sC$.
\end{itemize}
Moreover, let $B(\sC)$ be the faithful algebra $A/\ann_A(\sC)$ of $\sC$. Then $\sC$ is a~faithful component of $\Gamma_{B(\sC)}$ and
$\Hom_A(X,Y)=\Hom_{B(\sC)}(X,Y)$ for any modules $X$ and $Y$ in $\sC$. By general theory, we know that $\sC$ is a~generalized standard component in $\Gamma_A$
if and only if $\sC$ is a~generalized standard component in $\Gamma_{B(\sC)}$. Therefore, we may assume that $A=B(\sC)$.

We first prove that (i) implies (ii).
Assume that $\sC$ is generalized standard. The any nonzero homomorphism $f:X\to Y$ between two indecomposable modules $X$ and $Y$ in $\sC$
is a~finite sum of compositions of irreducible homomorphisms between indecomposable modules in $\sC$
(see \cite[Proposition V.7.5]{ARS} or \cite[Proposition VII.3.9]{SY4}). In particular, $\rad_A(X,Y)\neq 0$, for $X$ and $Y$ in $\sC$, implies that
there is a~path in $\sC$ from $X$ to $Y$.
\begin{enumerate}
\renewcommand{\labelenumi}{\rm(\arabic{enumi})}
\item It follows from the definition of $M_{\sC}^{(l)}$ that there is no path in $\sC$ from an indecomposable direct summand of
$P_{\sC}\oplus M_{\sC}^{(t)}\oplus M_{\sC}^{(r)}$ to an indecomposable direct summand of $M_{\sC}^{(l)}$, and hence
$\Hom_A(P_{\sC}\oplus M_{\sC}^{(t)}\oplus M_{\sC}^{(r)},M_{\sC}^{(l)})=0$. Similarly, it follows from the definition of $M_{\sC}^{(r)}$ that there
is no path in $\sC$ from an indecomposable direct summand of $M_{\sC}^{(r)}$ to an indecomposable direct summand of
$Q_{\sC}\oplus M_{\sC}^{(t)}\oplus M_{\sC}^{(l)}$, and hence
$\Hom_A(M_{\sC}^{(r)},Q_{\sC}\oplus M_{\sC}^{(t)}\oplus M_{\sC}^{(l)})=0$.
\item For a~component $\Gamma$ of $_l\sC^{\ac}$, the translation quiver $\sD_{\Gamma}^{(l)}$ is generalized standard, and then
$\Hom_A(M_{\Gamma}^{(l)},\tau_AM_{\Gamma}^{(l)})=0$, by Proposition \ref{prop-D-left}. Similarly, for a~component $\Gamma$ of $_r\sC^{\ac}$,
the translation quiver $\sD_{\Gamma}^{(r)}$ is generalized standard, and then
$\Hom_A(\tau_A^{-1}M_{\Gamma}^{(r)},M_{\Gamma}^{(r)})=0$, by Proposition \ref{prop-D-left-dual}.
Therefore, we conclude that 
\[ \Hom_A(M_{\sC}^{(l)},\tau_AM_{\sC}^{(l)})=0 \,\,\,\text{and}\,\,\, \Hom_A(\tau_A^{-1}M_{\sC}^{(r)},M_{\sC}^{(r)})=0. \]
\item Let $\Gamma$ be a~component of $_c\sC^{\coh}$ and $\cT_{\Gamma}$ the tubular part of $\Gamma$. Since $\sC$ is generalized standard,
we conclude that $\Gamma$ is also generalized standard. Then it follows from Proposition \ref{prop-genst} that $\cT_{\Gamma}$ is a~finite
faithful family of pairwise orthogonal generalized standard stable tubes in $\Gamma_{B(\cT_{\Gamma})}$ and $B(\Gamma)$ is a~generalized
multicoil enlargement of $B(\cT_{\Gamma})$ with respect to the family $\cT_{\Gamma}$.
In particular, applying Proposition \ref{prop-brick}, we obtain that $\rad_A(M_{\Gamma}^{(t)},M_{\Gamma}^{(t)})=0$.
Moreover, the simple composition factors of $M_{\Gamma}^{(t)}$ are not direct summands of $\tom(P_{\sC})\oplus\soc(Q_{\sC})$, and hence
$\Hom_A(P_{\sC},M_{\Gamma}^{(t)})=0$ and $\Hom_A(M_{\Gamma}^{(t)},Q_{\sC})=0$. We also note that, for two different components
$\Gamma$ and $\Omega$ of $_c\sC^{\coh}$, the tubular parts $\cT_{\Gamma}$ and $\cT_{\Omega}$ are disjoint and orthogonal, by
Propositions \ref{prop-apcoh} and \ref{prop-genst}, and hence
$\rad_A(M_{\Gamma}^{(t)},M_{\Omega}^{(t)})=0$ and $\rad_A(M_{\Omega}^{(t)},M_{\Gamma}^{(t)})=0$. Summing up we have
\[ \Hom_A(P_{\sC},M_{\sC}^{(t)})=0,\hspace{3mm} \Hom_A(M_{\sC}^{(t)},Q_{\sC})=0,\hspace{3mm} \rad_A(M_{\sC}^{(t)},M_{\sC}^{(t)})=0. \]
\item We shall prove that $\Hom_A(\tom(P_{\sC}^{\coh}),Q_{\sC})=0$. Let $S=\tom(P)$ for an indecomposable direct summand $P$ of $P_{\sC}^{\coh}$.
Assume first that $S$  does not belong to $\sC$. Since $\sC$ is generalized standard and there is a~canonical epimorphism $P\to S$,
we conclude that $\Hom_A(S,\soc(Q_{\sC}))=\Hom_A(S,Q_{\sC})=0$. Therefore, assume that $S$ belongs to $\sC$. Since $P$ is right coherent,
there exists an infinite sectional path
\[P = X_0 \longrightarrow X_1 \longrightarrow \cdots \longrightarrow X_i\longrightarrow X_{i+1} \longrightarrow X_{i+2} \longrightarrow \cdots \]
with all but finitely many modules from a~cyclic generalized multicoil $\Gamma$ of $_c\sC^{\coh}$ (see Proposition \ref{prop-apcoh}).
Moreover, by Propositions \ref{prop-apcoh} and \ref{prop-genst}, $B(\Gamma)$ is a~generalized multicoil enlargement of $B(\cT_{\Gamma})$
with respect to the tubular family $\cT_{\Gamma}$ of $\Gamma$, and $\Gamma$ is the cyclic part of the generalized multicoil $\Omega$,
created from $\cT_{\Gamma}$ by iterated application of translation quiver admissible operations corresponding to the algebra admissible operations
leading from $B(\cT_{\Gamma})$ to $B(\Gamma)$.
We have also $B(\Gamma)=B(\Omega)$. By Proposition \ref{prop-gme} we may obtain $B(\Omega)$ from $B(\cT_{\Gamma})$ applying first a~sequence
of admissible operations of type (ad~1), creating a~tubular extension $B$ of $B(\cT_{\Gamma})$, and then applying a~sequence
of admissible operations of types (ad~1$^*$)-(ad~5$^*$). Then the projective module $P$ and its top $S$ are acyclic modules in one of the
created ray tubes of $\Gamma_B$. In particular, we conclude that the injective envelope $E_{B(\Omega)}(S)$ of $S$ in $\mo B(\Omega)$
does not belong to $\Omega$, and consequently $\rad_{B(\Omega)}^{\infty}(S,E_{B(\Omega)}(S))\neq 0$.
Since $B(\Omega)$ is a~quotient algebra of $A$, we conclude that $E_{B(\Omega)}(S)$ is a~submodule of an injective envelope $E_A(S)$ of $S$
in $\mo A$, and hence $\rad_A^{\infty}(S,E_A(S))\neq 0$.
This shows that $E_A(S)$ does not belong to $\sC$, because $\sC$ is generalized standard. Hence $\Hom_A(S,Q_{\sC})=0$.
Therefore, we obtain $\Hom_A(\tom(P_{\sC}^{\coh}),Q_{\sC})=0$.
\item Applying dual arguments to those in (4), we conclude that $\Hom_A(P_{\sC},\soc(Q_{\sC}^{\coh}))=0$.
\end{enumerate}
This finishes the proof of the implication (i) $\Longrightarrow$ (ii).

Assume now that the conditions of (ii) are satisfied. We shall show that $\sC$ is generalized standard in few steps.
We use the notation introduced in Section \ref{sec2}.
Moreover, we denote by $\sD_{\sC}^{(l)}$ the disjoint union of the translation quivers $\sD_{\Gamma}^{(l)}$ for all components $\Gamma$
of $_l\sC^{\ac}$, and by $\sD_{\sC}^{(r)}$ the disjoint union of the translation quivers $\sD_{\Gamma}^{(r)}$ for all components $\Gamma$
of $_r\sC^{\ac}$.
%
\begin{enumerate}
\renewcommand{\labelenumi}{\rm(\alph{enumi})}
\item Let $\Gamma$ be a~component of $_l\sC^{\ac}$. Since $\Hom_A(M_{\sC}^{(l)},\tau_AM_{\sC}^{(l)})=0$, we have
$\Hom_A(M_{\Gamma}^{(l)},\tau_AM_{\Gamma}^{(l)})$ $=0$ and hence the translation quiver $\sD_{\Gamma}^{(l)}$ is generalized standard,
by Proposition \ref{prop-D-left}.
\item Let $\Gamma$ be a~component of $_r\sC^{\ac}$. Since $\Hom_A(\tau_A^{-1}M_{\sC}^{(r)},M_{\sC}^{(r)})=0$, we have 
 $\Hom_A(\tau_A^{-1}M_{\Gamma}^{(r)},$ $M_{\Gamma}^{(r)})=0$ and hence the translation quiver $\sD_{\Gamma}^{(r)}$ is generalized standard,
by Proposition \ref{prop-D-left-dual}.
\item Let $\Gamma$ be a~component of $_c\sC^{\coh}$. Then the condition $\rad_A(M_{\sC}^{(t)},M_{\sC}^{(t)})=0$ implies 
$\rad_A(M_{\Gamma}^{(t)},$ $M_{\Gamma}^{(t)})=0$. Let $\cT_{\Gamma}$ be the tubular family of $\Gamma$, $C=B(\cT_{\Gamma})$ and
$B=B(\Gamma)$. Since $\Hom_A(P_{\sC},M_{\Gamma}^{(t)})=0$ (respectively, $\Hom_A(M_{\Gamma}^{(t)},Q_{\sC})=0$), we have
$\Hom_A(P_{\sC},Z)=0$ (respectively, $\Hom_A(Z,Q_{\sC})$ $=0$) for any indecomposable module $Z$ in $\cT_{\Gamma}$. Hence $\cT_{\Gamma}$
is a~faithful family of stable tubes in $\Gamma_C$ with $\rad_C(M_{\Gamma}^{(t)},M_{C}^{(t)})=0$. Then it follows from Proposition \ref{prop-brick}
that $\cT_{\Gamma}$ is a~family of generalized standard stable tubes in $\Gamma_C$. Moreover, the tubes in $\cT_{\Gamma}$ are pairwise orthogonal.
Indeed, if $\Hom_C(U,V)\neq 0$ for two modules in different tubes, say ${\cT}_1$ containing $U$ and ${\cT}_2$ containing $V$, then
$\rad_C(M_1,M_2)=\Hom_C(M_1,M_2)\neq 0$ for some modules $M_1$ from the mouth of ${\cT}_1$ and $M_2$ from the mouth of ${\cT}_2$,
and this contradicts the assumption $\rad_A(M_{\sC}^{(t)},M_{\sC}^{(t)})=0$. Then it follows from Proposition \ref{prop-genst} that $\Gamma$
is a~generalized standard cyclic multicoil and $B$ is a~generalized multicoil enlargement of $C$ with respect to the family $\cT_{\Gamma}$.
Therefore, we conclude that $_c\sC^{\coh}$ consists of generalized standard components, which are cyclic generalized multicoils
(Proposition \ref{prop-apcoh}).
\item It follows from the arguments applied in the proofs of statements (iv) and (v) of Theorem \ref{thm2} that the translation quiver
$\sD_{\sC}^{(l)}\cup {_c\sC^{\coh}}  \cup \sD_{\sC}^{(r)}$ is a~cofinite subquiver of $\sC$.
Hence, if $N$ is the direct sum of all indecomposable modules of $\sC$ which do not belong to this translation subquiver, then
$\rad_A^l(N,N)=0$ for some positive integer $l$.
\item  Assume now that $\rad_A^{\infty}(X,Y)\neq 0$ for some indecomposable modules $X$ and $Y$ in $\sC$. It follows from \cite[Lemma 2.1(i)]{S7}
that there exist an infinite path
\[  X = X_0 \buildrel {f_1}\over {\hbox to 7mm{\rightarrowfill}} X_1 \buildrel {f_2}\over {\hbox to 7mm{\rightarrowfill}} \cdots \buildrel {f_{s-1}}\over {\hbox to 7mm{\rightarrowfill}}
X_{s-1} \buildrel {f_s}\over {\hbox to 7mm{\rightarrowfill}} X_s \buildrel {f_{s+1}}\over {\hbox to 7mm{\rightarrowfill}} \cdots \]
of irreducible homomorphisms between indecomposable modules in $\sC$ and homomorphisms $g_s\in\rad_A^{\infty}(X_s,Y)$, $s\geq 1$,
such that $g_sf_s\ldots f_1\neq 0$ for any $s\geq 1$.
Then it follows from (d) that there exists a~positive integer $m$ such that $X_s$ belongs to ${_c\sC^{\coh}} \cup \sD_{\sC}^{(r)}$ for any $s\geq m$.
Hence, we may assume that $X$ belongs to ${_c\sC^{\coh}} \cup \sD_{\sC}^{(r)}$.
Applying \cite[Lemma 2.1(ii)]{S7}, we conclude that there exist an infinite path
\[  \cdots \buildrel {h_{r+1}}\over {\hbox to 7mm{\rightarrowfill}} Y_r \buildrel {h_r}\over {\hbox to 7mm{\rightarrowfill}} Y_{r-1} \buildrel {h_{r-1}}\over {\hbox to 7mm{\rightarrowfill}}
\cdots \buildrel {h_{2}}\over {\hbox to 7mm{\rightarrowfill}} Y_{1} \buildrel {h_1}\over {\hbox to 7mm{\rightarrowfill}} Y_0 = Y \]
of irreducible homomorphisms between indecomposable modules in $\sC$ and homomorphisms $u_r\in\rad_A^{\infty}(X,Y_r)$, $r\geq 1$,
such that $h_1\ldots h_ru_r\neq 0$ for any $r\geq 1$.
Then it follows from (d) that there exists a~positive integer $n$ such that $Y_r$ belongs to $\sD_{\sC}^{(l)} \cup {_c\sC^{\coh}}$ for any $r\geq n$.
Hence, we may also assume that $Y$ belongs to $\sD_{\sC}^{(l)} \cup {_c\sC^{\coh}}$.
\item Assume that $X$ belongs to $\sD_{\sC}^{(r)}$ and $Y$ belongs to $\sD_{\sC}^{(l)}$.
Let $X$ belongs to $\sD_{\Gamma}^{(r)}$ and $Y$ belongs to $\sD_{\Omega}^{(l)}$ for some components $\Gamma$ of $_r\sC^{\ac}$ and $\Omega$ of $_l\sC^{\ac}$.
Then there are an epimorphism $(M_{\Gamma}^{(r)})^a\to X$ and a~monomorphism $Y\to (M_{\Omega}^{(l)})^b$ for some positive integers $a, b$,
and $\rad_A^{\infty}(X,Y)\neq 0$ implies that $\Hom_A(M_{\Gamma}^{(r)},M_{\Omega}^{(l)})\neq 0$. But this contradicts the assumption
$\Hom_A(M_{\sC}^{(r)},M_{\sC}^{(l)}) = 0$.
\item Assume that $X$ belongs to $\sD_{\sC}^{(r)}$ and $Y$ belongs to ${_c\sC^{\coh}}$.
Let $X$ belongs to $\sD_{\Gamma}^{(r)}$ for a~component $\Gamma$ of ${_r\sC^{\ac}}$ and $Y$ belongs to a~component $\Omega$ of ${_c\sC^{\coh}}$.
It follows from (f) that we may also assume that $X$ belongs to $\Delta_{\Gamma}^{(r)}$. We know from \cite[Corollary B]{MS1} that all but
finitely many modules in the generalized multicoil $\Omega$ have exactly two direct predecessors and two direct successors.
Hence we may assume that there is an infinite sectional path
\[  Y = Z_0 \buildrel {v_1}\over {\hbox to 7mm{\rightarrowfill}} Z_1 \buildrel {v_2}\over {\hbox to 7mm{\rightarrowfill}} \cdots \buildrel {v_{i-1}}\over {\hbox to 7mm{\rightarrowfill}}
Z_{i-1} \buildrel {v_i}\over {\hbox to 7mm{\rightarrowfill}} Z_i \buildrel {v_{i+1}}\over {\hbox to 7mm{\rightarrowfill}} \cdots \]
of irreducible homomorphisms between indecomposable modules in $\Omega$ such that, for every $i\geq 1$, there is a~mesh complete translation
subquiver of the form
\[
\xymatrix@C=20pt@R=16pt{
\cdots\ar[r] & U_{j+1}^{(i)}\ar[r]\ar[d] & U_{j}^{(i)}\ar[r]\ar[d] & \cdots\ar[r] & U_{1}^{(i)}\ar[r]\ar[d] & U_{0}^{(i)}=Z_{i-1}\ar[d] \cr
\cdots\ar[r] & V_{j+1}^{(i)}\ar[r] & V_{j}^{(i)}\ar[r] & \cdots\ar[r] & V_{1}^{(i)}\ar[r] & V_{0}^{(i)}=Z_{i} \cr
}
\]
Then it follows from \cite[Corollary 1.6]{L1} (see also \cite[Corollary IX.3.17]{SY4}) that the homomorphisms $v_i$, $i\geq 1$,
are of infinite left degree, and hence $v_i\ldots v_1f\neq 0$ for any nonzero homomorphism $f: X\to Y$ and $i\geq 1$.
It follows from the structure of the generalized multicoil (see \cite[Section 2]{MS1} and \cite[Section 3]{MS2}) that infinitely many modules
$Z_i$ belong to the tubular part $\cT_{\Omega}$ of  $\Omega$. Hence $\Hom_A(X,Z)\neq 0$ for a~module $Z$ in $\cT_{\Omega}$. But then
$\Hom_A(X,W)\neq 0$ for a~module $W$ lying on the mouth of a~stable tube of $\cT_{\Omega}$.
Therefore, we have $\Hom_A(M_{\Gamma}^{(r)},M_{\Omega}^{(t)}) \neq 0$, which contradicts the condition $\Hom_A(M_{\sC}^{(r)},M_{\sC}^{(t)}) = 0$.
\item Assume that $X$ belongs to ${_c\sC^{\coh}}$ and $Y$ belongs to $\sD_{\sC}^{(l)}$. Then, applying dual arguments to those in (g), we
conclude that $\Hom_A(M_{\sC}^{(t)},M_{\sC}^{(l)})\neq 0$, again a~contradiction.
\item Assume now that the both $X$ and $Y$ belong to ${_c\sC^{\coh}}$. Then, taking into account (c), we conclude that there are different components
$\Gamma$ and $\Omega$ of ${_c\sC^{\coh}}$ such that $X$ belongs to $\Gamma$ and $Y$ belongs to $\Omega$. Then we may choose $X$ and $Y$
such that there exist a~nonzero homomorphism $f: X\to Y$, an infinite sectional path
\[  \cdots \buildrel {p_{j+1}}\over {\hbox to 7mm{\rightarrowfill}} W_j \buildrel {p_j}\over {\hbox to 7mm{\rightarrowfill}} W_{j-1} \buildrel {p_{j-1}}\over {\hbox to 7mm{\rightarrowfill}}
\cdots \buildrel {p_{2}}\over {\hbox to 7mm{\rightarrowfill}} W_{1} \buildrel {p_1}\over {\hbox to 7mm{\rightarrowfill}} W_0 = X \]
of irreducible homomorphisms of infinite right degree between indecomposable modules in $\Gamma$, and infinite sectional path
\[  Y = Z_0 \buildrel {v_1}\over {\hbox to 7mm{\rightarrowfill}} Z_1 \buildrel {v_2}\over {\hbox to 7mm{\rightarrowfill}} \cdots \buildrel {v_{i-1}}\over {\hbox to 7mm{\rightarrowfill}}
Z_{i-1} \buildrel {v_i}\over {\hbox to 7mm{\rightarrowfill}} Z_i \buildrel {v_{i+1}}\over {\hbox to 7mm{\rightarrowfill}} \cdots \]
of irreducible homomorphisms of infinite left degree between indecomposable modules in $\Omega$.
Then we have $v_i\ldots v_1fp_1\ldots p_j\neq 0$ for any $i,j\geq 1$. Since $\Gamma$ and $\Omega$ are generalized multicoils, there are $j\geq 1$ with $W_j$ in $\cT_{\Gamma}$
and $i\geq 1$ with $Z_i$ in $\cT_{\Omega}$. Thus $\Hom_A(W,Z)\neq 0$ for some $W\in\cT_{\Gamma}$ and $Z\in\cT_{\Omega}$, and consequently
$\rad_A(M_{\Gamma}^{(t)},M_{\Omega}^{(t)})\neq 0$, contradicting the condition $\rad_A(M_{\sC}^{(t)},M_{\sC}^{(t)}) = 0$.
\end{enumerate}
Summing up, we proved that $\sC$ is generalized standard.

We present now two example showing that the vanishing conditions (ii) of Theorem \ref{thm1} are necessary for the generalized standardness
of an Auslander-Reiten component.

\begin{ex} \label{ex-1}
Let $K$ be an algebraically closed field, $Q$ the quiver
of the form
\[
\xymatrix@C=16pt@R=20pt{
1\ar[rd]_{\gamma}&&2\ar@<2pt>[ll]^{\beta}\ar@<-2pt>[ll]_{\alpha}\cr
&3\ar[ru]_{\sigma}\cr
}
\]
and $I$ the ideal in the path algebra $KQ$ of $Q$ over $K$ generated by the paths $\alpha\gamma$, $\gamma\sigma$, $\sigma\alpha$, and $A=KQ/I$.
Moreover, let $H$ be the path algebra $K\Delta$ of the subquiver $\Delta$ of $Q$ given by the vertices $1, 2$ and the arrows $\alpha, \beta$.
Then $A$ is a~$10$-dimensional $K$-algebra and $H$ is a~quotient algebra of $A$. Applying \cite[Section X.4]{SS1} and \cite[Section XVI.1]{SS2}
we conclude that the Auslander-Reiten quiver $\Gamma_A$ has a~disjoin union decomposition
\[ \Gamma_A = \sC\cup\cT\cup \left( \bigcup_{\lambda\in K}{\cT}_{\lambda}\right), \]
where
\begin{enumerate}
\renewcommand{\labelenumi}{\rm(\arabic{enumi})}
\item $\sC$ is the component of the form
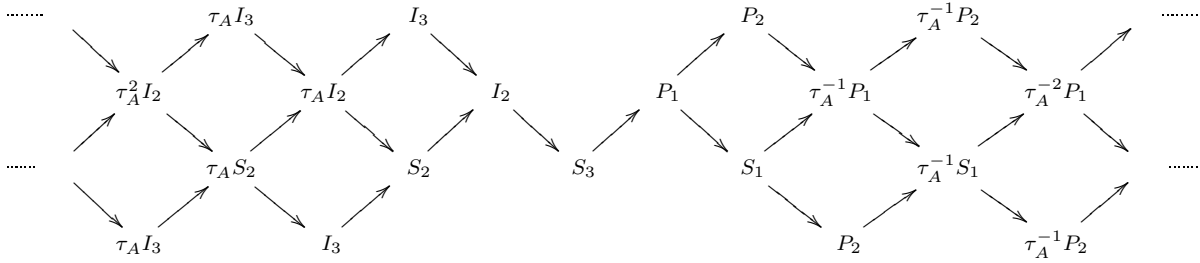
\[ \hspace{-10mm}
\xymatrix@C=10pt@R=12pt{
{\phantom{\scr{S}}}\ar@{.}[r] & {\phantom{\scr{S}}}\ar[rd] && \scr{\tau_AI_3}\ar[rd] && {\phantom{\tau}}\scr{I_3}\ar[rd] &&&& {\phantom{\tau}}\scr{P_2}\ar[rd] && \scr{\tau_A^{-1}P_2}\ar[rd] && {\phantom{\scr{S}}}\ar@{.}[r] & {\phantom{\scr{S}}} \cr
&& \scr{\tau_A^{2}I_2}\ar[ru]\ar[rd] && \scr{\tau_AI_2}\ar[ru]\ar[rd] && {\phantom{\tau}}\scr{I_2}\ar[rd] && {\phantom{\tau}}\scr{P_1}\ar[ru]\ar[rd] && \scr{\tau_A^{-1}P_1}\ar[ru]\ar[rd] && \scr{\tau_A^{-2}P_1}\ar[ru]\ar[rd] \cr
{\phantom{\scr{S}}}\ar@{.}[r] & {\phantom{\scr{SS}}}\ar[rd]\ar[ru] && \scr{\tau_AS_2}\ar[rd]\ar[ru] && {\phantom{\tau}}\scr{S_2}\ar[ru] && {\phantom{\tau}}\scr{S_3}\ar[ru] && {\phantom{\tau}}\scr{S_1}\ar[rd]\ar[ru] && \scr{\tau_A^{-1}S_1}\ar[ru]\ar[rd] && {\phantom{\scr{SS}}}\ar@{.}[r] & {\phantom{\scr{S}}} \cr
&& \scr{\tau_AI_3}\ar[ru] && {\phantom{\tau}}\scr{I_3}\ar[ru] && && && {\phantom{\tau}}\scr{P_2}\ar[ru] && \scr{\tau_A^{-1}P_2}\ar[ru] \cr
}
 \]
containing all indecomposable postprojective modules and all indecomposable preinjective modules over the Kronecker algebra $H$;

\item $\cT$ is the quasi-tube of the form
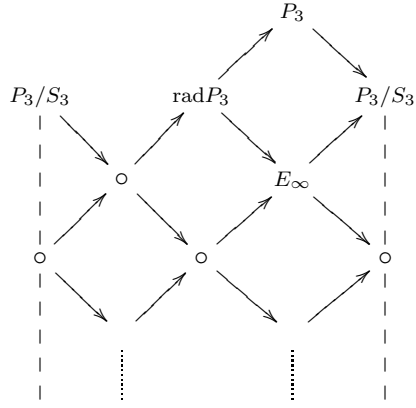
\[
\xymatrix@C=10pt@R=18pt{
&&&\scr{P_3}\ar[rd]\cr
\scr{P_3/S_3}\ar[rd]\ar@{--}[dd]&&\scr{\rad P_3}\ar[rd]\ar[ru]&&\scr{P_3/S_3}\ar@{--}[dd]\cr
&\circ\ar[ru]\ar[rd]&&\scr{E_{\infty}}\ar[ru]\ar[rd]\cr
\circ\ar[ru]\ar[rd]\ar@{--}[dd]&&\circ\ar[ru]\ar[rd]&&\circ\ar@{--}[dd]\cr
&{\phantom{\circ}}\ar[ru]\ar@{.}[d]&&{\phantom{\circ}}\ar[ru]\ar@{.}[d]\cr
{\phantom{\circ}}&{\phantom{\circ}}&{\phantom{\circ}}&{\phantom{\circ}}&{\phantom{\circ}}\cr
}
 \]
obtained from the stable tube ${\cT}_{\infty}^H$ of $\Gamma_H$ of rank $1$, with the mouth module
\[
E_{\infty}: \xymatrix@C=16pt@R=20pt{
K&&K,\ar@<2pt>[ll]^{1}\ar@<-2pt>[ll]_{0}\cr
}
\]
by one ray insertion and one coray insertion, and the modules along the dashed lines have to be identified;
\item For each $\lambda\in K$, ${\cT}_{\lambda}={\cT}_{\lambda}^H$ is the stable tube of rank $1$ in $\Gamma_H$
with the mouth module
\[
E_{\lambda}: \xymatrix@C=16pt@R=20pt{
K&&K.\ar@<2pt>[ll]^{\lambda}\ar@<-2pt>[ll]_{1}\cr
}
\]
We observe that the stable tubes ${\cT}_{\lambda}$, $\lambda\in K$, are generalized standard in $\mo A$. On the other hand, for the quasi-tube $\cT$
we have
\[ M_{\cT}^{(t)} = E_{\infty}, \,\,\,\, P_{\cT} = P_3 = P_{\cT}^{\coh}, \,\,\,\, Q_{\cT} = P_3 = Q_{\cT}^{\coh}. \]
Further, $P_3$ is an indecomposable projective-injective module with $\tom(P_3)$ $= S_3 = \soc(P_3)$, and $S_3$ lies in the component
$\sC$. Hence,
\[ \Hom_A(\tom(P_{{\cT}}^{\coh}),Q_{\cT}^{\coh}) = \Hom_A(\tom(P_{{\cT}}^{\coh}),\soc(Q_{\cT}^{\coh})) = \Hom_A(S_3,S_3)\neq 0. \]
Clearly, we have $\rad_A^{\infty}(P_3,P_3) = \rad_A(P_3,P_3)\neq 0$, and so $\cT$ is not generalized standard.
We also note that $\Hom_A(P_{\cT},M_{\cT}^{(t)})=0$ and $\Hom_A(M_{\cT}^{(t)},Q_{\cT})=0$.
Consider now the acyclic component $\sC$. We may take $M_{\sC}^{(l)} = S_2\oplus I_2\oplus I_3$ and $M_{\sC}^{(r)} = S_1\oplus P_1\oplus P_2$.
We have
\[ \Hom_A(P_1,I_2) = \rad_A^{\infty}(P_1,I_2)\neq 0, \]
and hence $\Hom_A(M_{\sC}^{(r)},M_{\sC}^{(l)})\neq 0$ and $\sC$ is not generalized standard.
\end{enumerate}
\end{ex}
\begin{ex} \label{ex-2}
Let $K$ be an algebraically closed field, $a\in K\setminus\{0,1\}$, and $R$ be the locally bounded $K$-category given by the infinite quiver
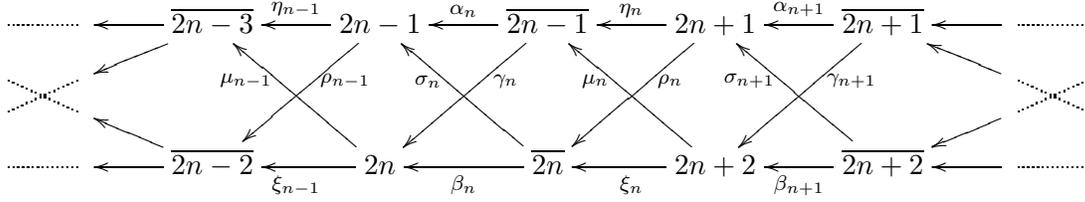
\[
\xymatrix@C=26pt@R=8pt{
{\phantom{}}\ar@{.}[r] & {\phantom{}} & \overline{2n-3}\ar[l]\ar[ld] & 2n-1\ar[l]_{\eta_{n-1}}\ar[lddd]_{\mu_{n-1}\,\,\,\,\,\,}
& \overline{2n-1}\ar[l]_{\alpha_{n}}\ar[lddd]_{\sigma_{n}\,\,\,\,\,} & 2n+1\ar[l]_{\eta_{n}}\ar[lddd]_{\mu_{n}\,\,\,\,\,}
& \overline{2n+1}\ar[l]_{\alpha_{n+1}}\ar[lddd]_{\sigma_{n+1}\,\,\,\,\,\,} & {\phantom{}}\ar[l] & {\phantom{}}\ar@{.}[l] \cr
{\phantom{}}\ar@{.}[rd] & {\phantom{}} &&&&&& {\phantom{}}\ar[lu] & {\phantom{}}\ar@{.}[ld] \cr
{\phantom{}}\ar@{.}[ru] & {\phantom{}} &&&&&& {\phantom{}}\ar[ld] & {\phantom{}}\ar@{.}[lu] \cr
{\phantom{}}\ar@{.}[r] & {\phantom{}} & \overline{2n-2}\ar[l]\ar[lu] & 2n\ar[l]^{\xi_{n-1}}\ar[luuu]_{\,\,\,\,\,\,\rho_{n-1}}
& \overline{2n}\ar[l]^{\beta_{n}}\ar[luuu]_{\,\,\,\,\,\,\,\gamma_{n}} & 2n+2\ar[l]^{\xi_{n}}\ar[luuu]_{\,\,\,\,\,\,\rho_{n}}
& \overline{2n+2}\ar[l]^{\beta_{n+1}}\ar[luuu]_{\,\,\,\,\,\,\,\gamma_{n+1}} & {\phantom{}}\ar[l] & {\phantom{}}\ar@{.}[l] \cr
}
\]
and the relations $\eta_n\alpha_n=\rho_n\sigma_n$, $\xi_n\beta_n=a\mu_n\gamma_n$, $\alpha_{n+1}\eta_n=\gamma_{n+1}\mu_n$,
$\beta_{n+1}\xi_n=a\sigma_{n+1}\rho_n$, $\eta_n\alpha_n\eta_{n-1}=0$, $\alpha_{n+1}\eta_n\alpha_n=0$,
$\xi_n\beta_n\xi_{n-1}=0$, $\beta_{n+1}\xi_n\beta_n=0$,
$\rho_n\sigma_n\rho_{n-1}=0$, $\sigma_{n+1}\rho_n\sigma_n=0$,
$\mu_n\gamma_n\mu_{n-1}=0$, $\gamma_{n+1}\mu_n\gamma_n=0$,
$\eta_n\gamma_n\xi_{n-1}=0$, $\gamma_{n+1}\xi_n\sigma_n=0$,
$\xi_n\sigma_n\eta_{n-1}=0$, $\sigma_{n+1}\eta_n\gamma_n=0$,
for all $n\in{\Bbb Z}$. Moreover, let $G$ be the infinite cyclic group of automorphisms of $R$ generated by the shift $g$ given by
$g(2n-1)=2n+1$, $g(2n)=2n+2$, $g(\overline{2n-1})=\overline{2n+1}$, $g(\overline{2n})=\overline{2n+2}$, $g(\alpha_n)=\alpha_{n+1}$,
$g(\beta_n)=\beta_{n+1}$, $g(\gamma_n)=\gamma_{n+1}$, $g(\sigma_n)=\sigma_{n+1}$, $g(\eta_n)=\eta_{n+1}$, $g(\xi_n)=\xi_{n+1}$, $g(\rho_n)=\rho_{n+1}$,
for any $n\in{\Bbb Z}$. Consider now the orbit algebras $A=R/G$ and the push-down functor $F_{\lambda}: \mo R\to\mo A$ associated to the Galois
covering $F: R\to R/G=A$. Since the group $G$ is torsion-free, it follows from \cite{Ga} that $F_{\lambda}$ preserves indecomposable modules and almost split
sequences. The orbit algebra $A$ is given by the quiver
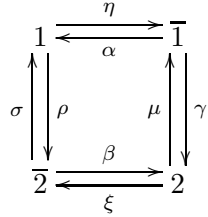
\[
\xymatrix@C=40pt@R=40pt{
1\ar@<5pt>[r]^{\eta}\ar@<3pt>[d]^{\rho}&\overline{1}\ar[l]^{\alpha}\ar@<3pt>[d]^{\gamma}\cr
\overline{2}\ar@<5pt>[r]^{\beta}\ar@<3pt>[u]^{\sigma}&2\ar[l]^{\xi}\ar@<3pt>[u]^{\mu}\cr
}
\]
and the relations $\eta\alpha=a\rho\sigma$, $\xi\beta=\mu\gamma$, $\alpha\eta=a\gamma\mu$, $\beta\xi=\sigma\rho$,
$\eta\alpha\eta=0$, $\alpha\eta\alpha=0$, $\xi\beta\xi=0$, $\beta\xi\beta=0$, $\rho\sigma\rho=0$, $\sigma\rho\sigma=0$, $\mu\gamma\mu=0$,
$\gamma\mu\gamma=0$, $\eta\gamma\xi=0$, $\gamma\xi\sigma=0$, $\xi\sigma\eta=0$, $\sigma\eta\gamma=0$.
We denote by $B$ the full subcategory of $R$ given by the convex subquiver
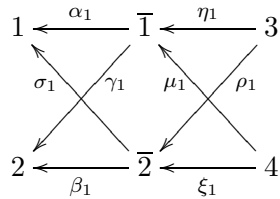
\[
\xymatrix@C=36pt@R=8pt{
 1 & \overline{1}\ar[l]_{\alpha_{1}}\ar[lddd]_{\sigma_{1}\,\,\,\,} & 3\ar[l]_{\eta_{1}}\ar[lddd]_{\mu_{1}\,\,\,} \cr
&& \cr
&& \cr
 2 & \overline{2}\ar[l]^{\beta_{1}}\ar[luuu]_{\,\,\,\,\gamma_{1}} & 4\ar[l]^{\xi_{1}}\ar[luuu]_{\,\,\,\,\,\rho_{1}} \cr
}
\]
Then $B$ is a~tubular algebra of tubular type $(2,2,2,2)$ in the sense of \cite{Ri1}. Then it follows from \cite[5.2]{Ri1} that $\Gamma_B$ admits
a~generalized standard ray tube $\sC$ and a~generalized standard coray tube $\sD$ of the forms
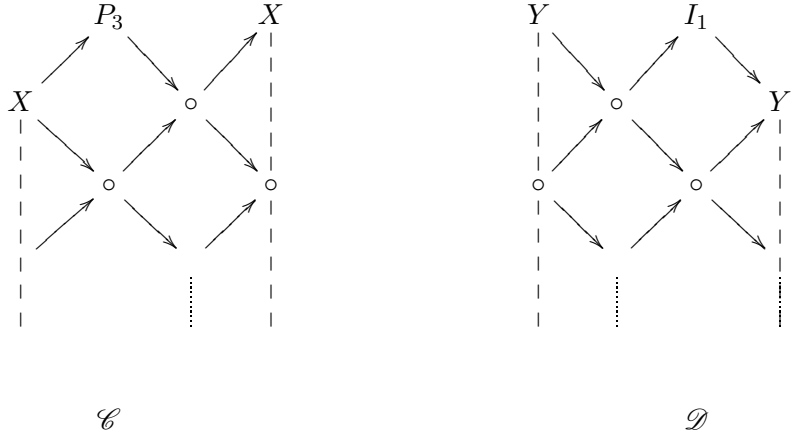
\[
\xymatrix@C=16pt@R=18pt{
&&P_3\ar[rd]&&X\ar@{--}[dd]\cr
&X\ar[ru]\ar[rd]\ar@{--}[ddd]&&\circ\ar[ru]\ar[rd]\cr
&&\circ\ar[ru]\ar[rd]&&\circ\ar@{--}[dd]\cr
&{\phantom{\circ}}\ar[ru]&&{\phantom{\circ}}\ar[ru]\ar@{.}[d]\cr
{\phantom{\circ}}&{\phantom{\circ}}&{\phantom{\circ}}&{\phantom{\circ}}&{\phantom{\circ}}\cr
&&\sC\cr
}
\hspace{30mm}
\xymatrix@C=16pt@R=18pt{
Y\ar[rd]\ar@{--}[dd]&&I_1\ar[rd]&&\cr
&\circ\ar[ru]\ar[rd]&&Y\ar@{--}[ddd]\cr
\circ\ar[ru]\ar[rd]\ar@{--}[dd]&&\circ\ar[ru]\ar[rd]&&\cr
&{\phantom{\circ}}\ar[ru]\ar@{.}[d]&&{\phantom{\circ}}\ar@{.}[d]\cr
{\phantom{\circ}}&{\phantom{\circ}}&{\phantom{\circ}}&{\phantom{\circ}}&{\phantom{\circ}}\cr
&&\sD\cr
}
 \]
where $X$ and $Y$ are the indecomposable $B$-modules
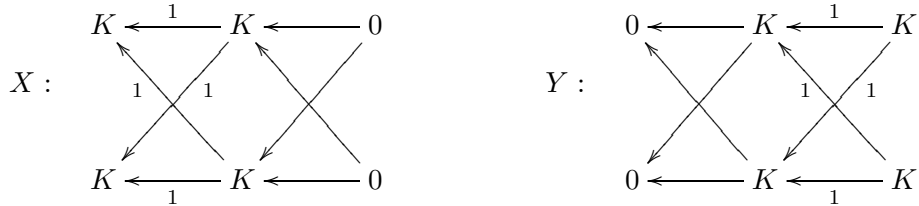
\[
\xymatrix@C=36pt@R=8pt{
& K & K\ar[l]_{1}\ar[lddd]_{1\,\,\,\,\;} & 0\ar[l]\ar[lddd] \cr
X:\hspace{-14mm}&&& \cr
&&& \cr
& K & K\ar[l]^{1}\ar[luuu]_{\;\,\,\,\,1} & 0\ar[l]\ar[luuu] \cr
}
\hspace{15mm}
\xymatrix@C=36pt@R=8pt{
& 0 & K\ar[l]\ar[lddd] & K\ar[l]_{1}\ar[lddd]_{1\,\,\,\,} \cr
Y:\hspace{-14mm}&&& \cr
&&& \cr
& 0 & K\ar[l]\ar[luuu] & K\ar[l]^{1}\ar[luuu]_{\,\,\,\,\,\,1} \cr
}
\]
In particular, we have $M_{\sC}^{(t)}=X$ and $M_{\sD}^{(t)}=Y$. Observe that $\Hom_B(P_{\sC},M_{\sC}^{(t)})$ $= \Hom_B(P_3,X)=0$
and $\Hom_B(M_{\sD}^{(t)},Q_{\sC}) = \Hom_B(Y,I_1)=0$.
Since $\sC$ and $\sD$ are components in $\Gamma_R$, we have the induced ray tube $\sC^*=F_{\lambda}(\sC)$, containing the indecomposable
projective module $F_{\lambda}(P_3)=F_{\lambda}(P_1)$, and coray tube $\sD^*=F_{\lambda}(\sD)$, containing the indecomposable
injective module $F_{\lambda}(I_1)=F_{\lambda}(I_3)$. Hence, we obtain
$P_{\sC^*} = P_{\sC^*}^{\coh} = F_{\lambda}(P_1)$ and $Q_{\sD^*} = Q_{\sD^*}^{\coh} = F_{\lambda}(I_3)$.
Moreover, $M_{\sC^*}^{(t)} = F_{\lambda}(X)$ and $M_{\sD^*}^{(t)} = F_{\lambda}(Y)$, with
$\rad_A(F_{\lambda}(X),F_{\lambda}(X))=\rad_A(X,X)=0$ and $\rad_A(F_{\lambda}(Y),F_{\lambda}(Y))=\rad_A(Y,Y)=0$.
We also note that there are nonzero homomorphisms $F_{\lambda}(P_1)\to F_{\lambda}(S_1)\to F_{\lambda}(X)$ and
$F_{\lambda}(Y)\to F_{\lambda}(S_1)\to F_{\lambda}(I_1)$ and the simple module $F_{\lambda}(S_1)$ is neither in $\sC^*$ nor in $\sD^*$.
Therefore, $\sC^*$, $\sD^*$ are not generalized standard components of $\Gamma_A$, and
$\Hom_A(P_{\sC^*},M_{\sC^*}^{(t)})\neq 0$, $\Hom_A(M_{\sD^*}^{(t)},Q_{\sD^*})\neq 0$.
\end{ex}

\section{Proof of Theorem \ref{thm7}} \label{sec5} 

The aim of this section is to prove Theorem \ref{thm7} and recall the relevant facts.

Following Happel, Reiten and Smal{\o} \cite{HRS} an algebra $B$ is called quasitilted if $B$ is the endomorphism algebra $\End_{\sH}(T)$ of a~tilting object $T$
in an~abelian hereditary category $\sH$. In the case when $\sH$ is equivalent (as a~triangulated category) to the derived category $D^b(\mo\Lambda)$ of the
module category $\mo\Lambda$ of a~canonical algebra $\Lambda$ (in the sense of Ringel \cite{Ri1, Ri2}), $B=\End_{\sH}(T)$ is called a~quasitilted algebra
of canonical type. It was shown in \cite{HRS} that an algebra $B$ is quasitilted if and only if $\gldim B\leq 2$ and every module $X$ in $\ind B$ satisfies
$\pd_BX\leq 1$ or $\id_BX\leq 1$. In \cite{HRe} Happel and Reiten proved that an algebra $A$ is quasitilted if and only if $A$ is tilted or quasitilted
of canonical type.

The prominent class of quasitilted algebras of canonical type is formed by concealed canonical algebras \cite{LM}, which are the postprojective tilts
of canonical algebras. In \cite{LP} Lenzing and de la Pe\~na proved that the concealed canonical algebras form the class of all algebras whose
Auslander-Reiten quiver admits a~separating family of stable tubes. Then we have the following characterization of arbitrary quasitilted algebras
of canonical type established by Lenzing and Skowro\'nski \cite{LS}.
\begin{prop}\label{prop-qtact}%
Let $A$ be an algebra. The following statements are equivalent:
\begin{enumerate}
\renewcommand{\labelenumi}{\rm(\roman{enumi})}
\item $A$ is a~quasitilted algebra of canonical type.
\item $\Gamma_A$ admits a~separating family of semiregular tubes (ray tubes and coray tubes).
\item $A$ is a~semiregular branch enlargement of a~concealed canonical algebra.
\end{enumerate}
\end{prop}

It follows that the quasitilted algebras of canonical type are very special generalized multicoil enlargements of concealed canonical algebras, invoking only
admissible operations of types (ad~1) and (ad~1$^*$), and applied to disjoint families of stable tubes.
We refer to \cite{HRS, LS, MS4, Ri1, Ri2} for basic representation theory of quasitilted algebras of canonical type. We note that the separating family
of semiregular tubes in the Auslander-Reiten quiver of a~quasitilted algebra of canonical type contains infinitely many stable tubes.

In the paper, by a~generalized multicoil algebra, we mean a~generalized multicoil enlargement $B$ of a~product $C$ of quasitilted algebras of canonical type
with respect to a~family of stable tubes of $\Gamma_C$. We have the following characterization of generalized multicoil algebras established
in \cite[Theorem A]{MS2}.
\begin{prop}\label{prop-gma}%
For an algebra $B$ the following statements are equivalent:
\begin{enumerate}
\renewcommand{\labelenumi}{\rm(\roman{enumi})}
\item $B$ is a~generalized multicoil algebra.
\item $\Gamma_B$ admits a~separating family of generalized multicoils.
\end{enumerate}
\end{prop}

The following consequences of \cite[Theorem C and Corollary D]{MS2} describe the structure of the Auslander-Reiten quiver of a~generalized multicoil algebra.
\begin{prop}\label{prop-gma-str}%
Let $B$ be a~generalized multicoil algebra, $\sC^B$ a~separating family of generalized multicoils in $\Gamma_B$, and $\Gamma_B={\cP}^B\cup{\sC}^B\cup{\cQ}^B$
the associated decomposition of $\Gamma_B$. Then the following statements hold.
\begin{enumerate}
\renewcommand{\labelenumi}{\rm(\roman{enumi})}
\item There is a~unique quotient algebra $B^{(l)}$ of $B$ which is a~quasitilted algebra of canonical type having a~separating family ${\cT}^{B^{(l)}}$ of coray tubes
 in $\Gamma_{B^{(l)}}$ such that $\Gamma_{B^{(l)}}={\cP}^{B^{(l)}}\cup{\cT}^{B^{(l)}}\cup{\cQ}^{B^{(l)}}$ with ${\cP}^{B^{(l)}}={\cP}^{B}$.
\item There is a~unique quotient algebra $B^{(r)}$ of $B$ which is a~quasitilted algebra of canonical type having a~separating family ${\cT}^{B^{(r)}}$ of ray tubes
 in $\Gamma_{B^{(r)}}$ such that $\Gamma_{B^{(r)}}={\cP}^{B^{(r)}}\cup{\cT}^{B^{(r)}}\cup{\cQ}^{B^{(r)}}$ with ${\cQ}^{B^{(r)}}={\cQ}^{B}$.
\item Every component in ${\cP}^{B}={\cP}^{B^{(l)}}$ is either a~postprojective component, a~ray tube, or obtained from a~component of the form
 ${\Bbb Z}{\Bbb A}_{\infty}$ by a~finite number (possibly zero) of ray insertions.
\item Every component in ${\cQ}^{B}={\cQ}^{B^{(r)}}$ is either a~preinjective component, a~coray tube, or obtained from a~component of the form
 ${\Bbb Z}{\Bbb A}_{\infty}$ by a~finite number (possibly zero) of coray insertions.
\end{enumerate}
\end{prop}

Then $B^{(l)}$ is called the \emph{left quasitilted algebra} of $B$ and $B^{(r)}$ the \emph{right quasitilted algebra} of $B$.

The following results from \cite[Theorem E]{MS2} describes homological properties of generalized multicoil algebras.

\begin{prop}\label{prop-hom}%
Let $B$ be a~generalized multicoil algebra, $\sC^B$ a~separating family of generalized multicoils in $\Gamma_B$, and $\Gamma_B={\cP}^B\cup{\sC}^B\cup{\cQ}^B$
the associated decomposition of $\Gamma_B$. Then the following statements hold.
\begin{enumerate}
\renewcommand{\labelenumi}{\rm(\roman{enumi})}
\item $\pd_BX\leq 1$ for any module $X$ in ${\cP}^B$.
\item $\id_BX\leq 1$ for any module $X$ in ${\cQ}^B$.
\item $\pd_BX\leq 2$ and $\id_BX\leq 2$ for any module $X$ in ${\sC}^B$.
\item $\gldim B\leq 3$.
\end{enumerate}
\end{prop}

The following proposition is a~crucial ingredient for the proofs of Theorems \ref{thm7} and \ref{thm4}.
\begin{prop}\label{prop-cru}%
Let $A$ be an algebra, $\sC=(\sC_i)_{i\in I}$ a~separating family of components in $\Gamma_A$, and assume that the cyclic part $_c\sC$ of $\sC$ is infinite.
Moreover, let $\Gamma$ be a~component of the coherent part $_c\sC^{\coh}$ of $_c\sC$. Then the following statements hold.
\begin{enumerate}
\renewcommand{\labelenumi}{\rm(\roman{enumi})}
\item The faithful algebra $C=B(\cT_{\Gamma})$ of the tubular part $\cT_{\Gamma}$ of $\Gamma$ is a~product of concealed canonical algebras, and
$\cT_{\Gamma}$ is a~finite family of stable tubes in $\Gamma_C$.
\item The faithful algebra $B=B(\Gamma)$ of $\Gamma$ is a~generalized multicoil enlargement of $C$ with respect to $\cT_{\Gamma}$, and hence
a~generalized multicoil algebra.
\item Every stable tube of $\Gamma_C$ which is not in the tubular part $\cT_{\Omega}$ of a~component $\Omega$ of $_c\sC^{\coh}$ belongs to the
family $\sC$.
\end{enumerate}
In particular, the separating family $\sC$ contains infinitely many stable tubes, and hence $I$ is infinite.
\end{prop}
\begin{proof}
Let $\Gamma_A={\cP}\cup{\sC}\cup{\cQ}$ be the decomposition of $\Gamma_A$ induced by the separating family $\sC$. Let $\sC_i$ be a~component in $\sC$
containing $\Gamma$, that is, $\Gamma$ is a~component of $_c\sC^{\coh}_{i}$.
We note that $\sC_i$ is a~generalized standard, hence almost periodic, component of $\Gamma_A$. It follows from Proposition \ref{prop-genst} that
$\cT_{\Gamma}$ is a~faithful family of pairwise orthogonal generalized standard stable tubes in $\Gamma_C$, and $B$ is a~generalized multicoil enlargement
of $C$ with respect $\cT_{\Gamma}$.
We claim that $C$ is a~product of concealed canonical algebras, and consequently $B$ is a~generalized multicoil algebra.
By \cite[Theorem 3.1]{RS1} (see also \cite[Proposition 3.2]{JMS}) it is enough to show that the family $\cT_{\Gamma}$ has no external short paths in $\ind C$.
Suppose that $\cT_{\Gamma}$ admits an external short path in $\ind C$, that is, a~sequence $X\to Y\to Z$ of nonzero nonisomorphisms in $\ind C$ such that
$X$ and $Z$ belong to $\cT_{\Gamma}$ but $Y$ is not in $\cT_{\Gamma}$. Because $X$ and $Z$ belong to the separating family $\sC$, the module $Y$ must belong
to $\sC$. Moreover, $\sC$ is generalized standard, and then there is a~finite path of irreducible homomorphisms
\[  X = X_0 \buildrel {f_1}\over {\hbox to 7mm{\rightarrowfill}} X_1 \buildrel {f_2}\over {\hbox to 7mm{\rightarrowfill}} \cdots \buildrel {f_{m-1}}\over {\hbox to 7mm{\rightarrowfill}}
X_{m-1} \buildrel {f_m}\over {\hbox to 7mm{\rightarrowfill}} X_m = Y  \]
in $\mo A$ with $f_m\ldots f_1\neq 0$, and hence $Y$ belongs to the component $\sC_i$. We observe also that $Y$ does not belong to the cyclic
generalized multicoil $\Gamma$, because $B=B(\Gamma)$ is a~generalized multicoil enlargement of $C$  and $Y$ is a~$C$-module not lying in $\cT_{\Gamma}$.
Then we conclude that there is $s\in\{1,\ldots, m\}$ such that $X_s$ belongs to $\tau_A^{-1}\Delta_{\Gamma}^{(r)}$. But it follows from the structure of
a~generalized multicoil that all modules  in $\tau_A^{-1}\Delta_{\Gamma}^{(r)}$ have no simple composition factors from $\mo C$, and then $f_s\ldots f_1=0$.
This proves the claim.

We shall prove now the statement (iii). Without loss of generality, we may assume that $C$ is indecomposable.
Let $\cT^C=(\cT_{\lambda}^C)_{\lambda\in\Lambda}$ be the separating family of stable tubes in $\Gamma_C$ and
\[ \Gamma_C={\cP}^C\cup{\cT}^C\cup{\cQ}^C \]
the associated decomposition of $\Gamma_C$. It is known (see \cite{LP}) that $\cT^C$ has the strongly separation property: every homomorphism
from $\cP^C$ to $\cQ^C$ factors through the additive category $\add(\cT^C_{\lambda})$ for any $\lambda\in\Lambda$. Let $\cT_{\lambda}$ be a~stable tube
of the tubular family $\cT_{\Gamma}$ of a~component $\Gamma$ of $_c\sC^{\coh}$. For any indecomposable module $X$ in $\cP^C$ there exists a~nonzero
homomorphism $f: X\to Y$ with $Y$ in $\add(\cT_{\lambda})$, and hence $\rad_A^{\infty}(X,Y)=\rad_C^{\infty}(X,Y)\neq 0$.
Similarly, for any indecomposable module $Z$ in $\cQ^C$ there is a~nonzero homomorphism $g: V\to Z$ with $V$ in $\add(\cT_{\lambda})$,
and hence $\rad_A^{\infty}(V,Z)=\rad_C^{\infty}(V,Z)\neq 0$. This shows that all indecomposable modules from $\cP^C$ are contained in $\cP$ and
all indecomposable modules from $\cQ^C$ are contained in $\cQ$. Let $\Lambda'$ be the set of all $\xi\in\Lambda$ such that $\cT^C_{\xi}$ occurs in
the tubular family $\cT_{\Omega}$ of a~component $\Omega$ of $_c\sC^{\coh}$. We also note that if a~stable tube $\cT^C_{\mu}$ has an indecomposable module lying in
$\sC$, then all indecomposable modules of $\cT^C_{\mu}$ lie in $\sC$, because $\Hom_{\Lambda}(\sC,\cP)=0$ and $\Hom_{\Lambda}(\cQ,\cP)=0$, and hence
$\mu\in\Lambda'$. Hence the indecomposable modules of the tubes $\cT^C_{\xi}$, $\xi\in\Lambda'$, are indecomposable $C$-modules lying in $\sC$.
Suppose now that $\Lambda\neq \Lambda'$ and $\sigma\in\Lambda\setminus\Lambda'$. Since $\Hom_{\Lambda}(\cQ,\cP)=0$ and any two modules of $\cT^C_{\sigma}$
lie on a~common cycle in $\ind C$, we infer that either all modules of $\cT^C_{\sigma}$ lie in $\cP$ or all modules of $\cT^C_{\sigma}$ lie in $\cQ$.
We may assume that $\cT^C_{\sigma}$ consists of modules from $\cP$.
Take now a~module $U$ in $\cT^C_{\sigma}$ and a~nonzero homomorphism $h: U\to E$, where $E$ is an indecomposable injective $C$-module.
Since $E$ lies in $\cQ^C$, it lies also in $\cQ$. Then $h$ factors through a~module $W$ in $\add(\sC)$, and so there is a~nonzero homomorphism $p: U\to T$
with $T$ an indecomposable module in $\sC$. Then $p(U)$ is contained in the largest $C$-submodule of $T$.
But it follows from the structure of indecomposable modules in a~generalized multicoil of a~generalized multicoil algebra (see \cite[Section 3]{MS2}),
and the above discussion, that the largest $C$-submodules of indecomposable modules in $\sC$ are the indecomposable modules from
$\cT_{\xi}$, $\xi\in\Lambda'$. Then $p=0$, because different tubes in $\cT^C$ are orthogonal. Summing up, we have  $\Lambda=\Lambda'$ and the
statement (iii) holds.
\end{proof}

We need also the following facts on tilted algebras.
\begin{prop}\label{prop-tilted-1}%
Let $H$ be an indecomposable hereditary algebra, $T$ a~tilting module in $\mo H$ without indecomposable preinjective direct summands,
$B=\End_H(T)$ the associated tilted algebra, and $\sY\Gamma_B$ the full translation subquiver of $\,\Gamma_B$ formed by all components contained
entirely in the torsion part $\sY(T)=\{Y\in\mo B\, |\, \Tor_1^B(Y,T)=0\}$.
Then the following statements hold.
\begin{enumerate}
\renewcommand{\labelenumi}{\rm(\roman{enumi})}
\item If $H$ is of Euclidean type, then $\sY\Gamma_B$ consists of a~postprojective component $\cP_B$ and an infinite family of ray tubes,
and the faithful algebra $B(\cP_B)$ of $\cP_B$ is a~tame concealed algebra.
\item If $H$ is of wild type, then $\sY\Gamma_B$ consists of a~postprojective component $\cP_B$ and an infinite family of components obtained
from components of the form ${\Bbb Z}{\Bbb A}_{\infty}$ by a~finite number (possibly zero) of ray insertions,
and the faithful algebra $B(\cP_B)$ of $\cP_B$ is a~wild concealed algebra.
\end{enumerate}
\end{prop}
\begin{proof}
The statement (i) follows from the theory of tilted algebras of Euclidean type developed in \cite{CB0-1, DR2, Ri-001, Ri1}.
The statement (ii) follows from the representation theory of tilted algebras of wild type developed in \cite{Ke, Ke2, Ke3, Ri-001, Ri-002, Ri1-2, St}.
\end{proof}

We have also the dual proposition.
\begin{prop}\label{prop-tilted-2}%
Let $H$ be an indecomposable hereditary algebra, $T$ a~tilting module in $\mo H$ without indecomposable postprojective direct summands,
$B=\End_H(T)$ the associated tilted algebra, and $\sX\Gamma_B$ the full translation subquiver of $\,\Gamma_B$ formed by all components contained
entirely in the torsion-free part $\sX(T)=\{X\in\mo B\, |\, X\otimes_BT)=0\}$.
Then the following statements hold.
\begin{enumerate}
\renewcommand{\labelenumi}{\rm(\roman{enumi})}
\item If $H$ is of Euclidean type, then $\sX\Gamma_B$ consists of a~preinjective component $\cQ_B$ and an infinite family of coray tubes,
and the faithful algebra $B(\cQ_B)$ of $\cQ_B$ is a~tame concealed algebra.
\item If $H$ is of wild type, then $\sX\Gamma_B$ consists of a~preinjective component $\cQ_B$ and an infinite family of components obtained
from components of the form ${\Bbb Z}{\Bbb A}_{\infty}$ by a~finite number (possibly zero) of coray insertions,
and the faithful algebra $B(\cQ_B)$ of $\cQ_B$ is a~wild concealed algebra.
\end{enumerate}
\end{prop}

We are now able to complete the proof of Theorem \ref{thm7}.

Let $A$ be an algebra with a~separating family $\sC^A$ of components in $\Gamma_A$, and $\Gamma_A = \cP^A\cup\sC^A\cup\cQ^A$ the associated decomposition
of $\Gamma_A$. Let $\sC^A=(\sC_i)_{i\in I}$. Applying Theorem \ref{thm4} we may assume that $I$ is infinite. Since every regular generalized standard
component is a~stable tube, the family $\sC^A$ has a~disjoint decomposition
\[ \sC^A = \sC\cup\cT, \]
where $\sC$ is the finite family of all nonregular components in $\sC^A$ and $\cT$ consists of all stable tubes in $\sC^A$.
Consider the faithful algebra $C=B(\cT)$ of $\cT$. Then $\cT$ is a~faithful family of pairwise orthogonal generalized standard stable tubes in $\Gamma_C$.
Moreover, because $\cT$ is a~part of the separating family $\sC^A$ of $\Gamma_A$ and $C$ is a~quotient algebra of $A$, we conclude that $\cT$ is without
external short paths in $\ind C$, and hence $C$ is a~product $C_1\times\ldots\times C_m$ of concealed canonical algebras. Let $J$ be the finite subset
of $I$ such that $\sC=(\sC_i)_{i\in J}$. For each $i\in J$, we take the left tilted quotient algebra $A_{\sC_i}^{(lt)}$ and
the right tilted quotient algebra $A_{\sC_i}^{(rt)}$ of $A$ satisfying the statements (i) and (ii) of Theorem \ref{thm2} for the component $\sC_i$.
Then we define $A^{(lt)}$ as the product of all algebras $A_{\sC_i}^{(lt)}$, $i\in J$, and $A^{(rt)}$ as the product of all algebras $A_{\sC_i}^{(rt)}$, $i\in J$.
Further, for each $i\in J$, we choose the quotient algebra $A_{\sC_i}^{(c)}$ of $A$ satisfying the statement (iii) of Theorem \ref{thm2},
and denote by $A^{(c)}$ the product of these algebras $A_{\sC_i}^{(c)}$, $i\in J$.
Consider also the disjoint union $\cT_{\sC}$ of the tubular parts $\cT_{\Gamma}$ of all components $\Gamma$ of $_c\sC^{\coh}$. Then it follows from Proposition
\ref{prop-cru} that the faithful algebra $B(\cT_{\sC})$ of $\cT_{\sC}$ coincides with $C$ and $A^{(c)}$ is a~multicoil enlargement of $C$ using
the family of stable tubes $\cT_{\sC}$ of $\Gamma_C$. We note that since $\cT_{\sC}$ is the disjoint union,
$\cT_{\sC}\cup\cT$ is a~separating family of stable tubes in $\Gamma_C$.
We denote by $A^{(lc)}$ the left quasitilted algebra $(A^{(c)})^{(l)}$ of $A^{(c)}$ and by $A^{(rc)}$ the right quasitilted algebra $(A^{(c)})^{(r)}$ of $A^{(c)}$.
Finally, we set
\[ A^{(l)} = A^{(lt)}\times A^{(lc)} \,\,\,\,\text{and}\,\,\,\,  A^{(r)} = A^{(rt)}\times A^{(rc)}. \]
We note that $A^{(l)}$ and $A^{(r)}$ are quotient algebras of $A$.
It follows now from Theorem \ref{thm2} and Propositions \ref{prop-gma-str}, \ref{prop-tilted-1}, \ref{prop-tilted-2} that $A^{(l)}$ and $A^{(r)}$ are the algebras
satisfying the statements (i)-(iv) of Theorem \ref{thm7}.

\section{Proof of Theorem \ref{thm4}} \label{sec6} 

The aim of this section is to prove Theorem \ref{thm4} and recall the relevant facts.

Let $A$ be an algebra and $\sC$ a~component of $\Gamma_A$. Then $\sC$ is said to be almost acyclic if the cyclic part $_c\sC$ of $\sC$ is finite.
I was shown in \cite[Theorem 2.5]{RS3} that $\sC$ is almost acyclic if and only if $\sC$ admits a~multisection $\Delta$, which is a~full connected
valued subquiver of $\sC$ satisfying the following conditions:
\begin{enumerate}
\renewcommand{\labelenumi}{\rm(\arabic{enumi})}
\item $\Delta$ is almost acyclic.
\item $\Delta$ is convex in $\sC$.
\item For each $\tau_A$-orbit $\sO$ in $\sC$, we have $1 \leq | \Delta \cap {\sO} | < \infty$.
\item $| \Delta \cap {\sO} | = 1$ for all but finitely many $\tau_A$-orbits $\sO$ in $\sC$.
\item No proper full convex valued subquiver of $\Delta$ satisfies the conditions \rm{(1)}--\rm{(4)}.
\end{enumerate}
To a~multisection $\Delta$ of $\sC$ one associates the left border $\Delta_l$, the right border $\Delta_r$, and the core $\Delta_c$ being
a~finite subquiver of $\Delta$ containing the cyclic part $_c\sC$ of $\sC$ (see \cite[Section 3]{RS3} for definitions and properties of these quivers).
We also note that if $\sC$ is additionally almost periodic, then every multisection of $\sC$ is finite.

Following \cite{RS3}, an algebra $B$ is said to be a~\emph{generalized double tilted algebra} if the following conditions are satisfied:
\begin{enumerate}
\renewcommand{\labelenumi}{\rm(\arabic{enumi})}
\item $\Gamma_B$ admits a component $\sC$ with a~faithful multisection $\Delta$.
\item There exists a~tilted quotient algebra $B^{(l)}$ of $B$ (not necessarily indecomposable) such that the left border $\Delta_l$ of $\Delta$
is a~disjoint union of sections
of the connecting components of the indecomposable parts of $B^{(l)}$, and the category of all predecessors of $\Delta_l$ in $\ind B$ coincides with
the category of all predecessors of $\Delta_l$ in $\ind B^{(l)}$.
\item There exists a~tilted quotient algebra $B^{(r)}$ of $B$ (not necessarily indecomposable) such that the right border $\Delta_r$ of $\Delta$
is a~disjoint union of sections
of the connecting components of the indecomposable parts of $B^{(r)}$, and the category of all successors of $\Delta_r$ in $\ind B$ coincides with
the category of all successors of $\Delta_r$ in $\ind B^{(r)}$.
\end{enumerate}
Then $B^{(l)}$ is called the \emph{left tilted algebra} and $B^{(r)}$ the \emph{right tilted algebra} of $B$.

The following characterization of generalized double tilted algebras from \cite[Theorem 3.1]{RS3} is essential for the proof of Theorem \ref{thm4}.
\begin{prop}\label{prop-gdta}
Let $B$ be an algebra. Then the following statements are equivalent:
\begin{enumerate}
\renewcommand{\labelenumi}{\rm(\roman{enumi})}
\item $B$ is a generalized double tilted algebra.
\item $\Gamma_B$ admits a~faithful generalized standard almost acyclic component.
\end{enumerate}
\end{prop}

We note that the class of generalized double tilted algebras contains all tilted algebras and all algebras of finite representation type.
We refer to \cite{RS3} for basic properties and representation theory of arbitrary generalized double tilted algebras. We only mention that,
if $B$ is a~generalized double tilted algebra, then for all but finitely modules $X$ in $\ind B$ we have $\pd_BX\leq 1$ or $\id_BX\leq 1$
but $B$ may have arbitrary (finite or infinite) global dimension.

We may now complete the proof of Theorem  \ref{thm4}.
The equivalence of the statements (ii) and (iii) follows from Proposition \ref{prop-gdta} and definition of
a~generalized double tilted algebra
The implication (ii) $\Longrightarrow$ (i) is obvious. The following proposition shows that also (i) implies (iii).
\begin{prop}\label{prop-62}
Let $A$ be an algebra such that $\Gamma_A$ admits a~finite separating family $\sC$ of components. Then $A$ is a~generalized double tilted algebra.
\end{prop}
\begin{proof}
Let $\sC_1, \ldots, \sC_m$ be the components of $\sC$. It follows from Proposition \ref{prop-cru} that the cyclic part $_c\sC$ of $\sC$ is finite.
Hence $\sC_1, \ldots, \sC_m$ are almost cyclic components of $\Gamma_A$.
Therefore, it is enough to prove that $m=1$, by Proposition \ref{prop-gdta}.

Assume $m\geq 2$. We shall prove that this contradicts the imposed indecomposablity of $A$. For each $i\in\{1,\ldots,m\}$,
we may choose a~finite multisection $\Delta^{(i)}$ of $\sC_i$, because $\sC_i$ is generalized standard, and hence almost periodic.
Further, let $\Delta_l^{(i)}$ be the left border of $\Delta^{(i)}$, $\Delta_r^{(i)}$ the right border of $\Delta^{(i)}$, and
$\Delta_c^{(i)}$ the core of $\Delta^{(i)}$. Then it follows from \cite[Proposition 2.4]{RS3} that every indecomposable module
$X$ in $\sC_i$ is in $\Delta_c^{(i)}$ or is a~predecessor of $\Delta_l^{(i)}$ or is a~successor of $\Delta_r^{(i)}$.
We denote by $M_i^{(l)}$ the direct sum of all indecomposable modules lying on $\Delta_l^{(i)}$.
Moreover, let $M^{(l)} = M_1^{(l)}\oplus\ldots\oplus M_m^{(l)}$, and $B^{(l)}=B(M^{(l)})$, the faithful algebra of $M^{(l)}$.
We note that $\Hom_A(M_i^{(l)},\tau_AM_i^{(l)})=0$ for any $i\in\{1,\ldots,m\}$, because $\sC_i$ is generalized standard and
there is no path in $\sC_i$ from an indecomposable direct summand of $M_i^{(l)}$ to an indecomposable direct summand of $\tau_AM_i^{(l)}$.
Moreover, for any $i\neq j$ in $\{1,\ldots,m\}$, we have $\Hom_A(M_i^{(l)},\tau_AM_j^{(l)})=0$, because the components $\sC_i$ and $\sC_j$
are orthogonal. Therefore, we conclude that $\Hom_A(M^{(l)},\tau_AM^{(l)})=0$. Let $H^{(l)}=\End_A(M^{(l)})$ and $T^{(l)}=D(M^{(l)})$,
where $D$ is the standard duality on $\mo A$. Then applying arguments from the proof of
\cite[Theorem VIII.5.6]{ASS}, or \cite[Theorem VIII.7.7]{SY4}, we conclude that:
\begin{enumerate}
\renewcommand{\labelenumi}{\rm(\arabic{enumi})}
\item $H^{(l)}$ is a~hereditary algebra (not necessarily indecomposable),
\item $T^{(l)}$ is a~tilting module in $\mo H^{(l)}$,
\item There is a~canonical isomorphism of algebras $B^{(l)}\to\End_{H^{(l)}}(T^{(l)})$.
\end{enumerate}
We note that $B^{(l)}$ (respectively, $H^{(l)}$) is a~product of indecomposable algebras corresponding to the
connected components of $\Delta_l$. In particular, we conclude that
\[ B^{(l)} = B_1^{(l)}\times \ldots \times B_m^{(l)}, \]
where $B_i^{(l)}=B(M_i^{(l)})$ for $i\in\{1,\ldots,m\}$. For each $i\in\{1,\ldots,m\}$, let $P_i$ be the direct sum of all
indecomposable projective modules in $\mo B_i^{(l)}$ and $Q_i$ the direct sum of all projective modules in $\sC_i$
which are not predecessors of $\Delta_l^{(i)}$. Moreover, let $D_i=\End_A(Q_i)$ and $N_i=\Hom_A(P_i,Q_i)$, for any
$i\in\{1,\ldots,m\}$. We note that $B_i^{(l)}=\End_A(P_i)$, and $N_i$ is a $B_i^{(l)}-D_i$-bimodule.
Further, since $M_i^{(l)}$ is a~faithful $B_i^{(l)}$-module, there is a~monomorphism $P_i\to(M_i^{(l)})^{n_i}$ for some $n_i\geq 1$.
Since $\sC_i$ is generalized standard, we have $\Hom_A(Q_i,M_i^{(l)})=0$, by the choice of $Q_i$, and hence $\Hom_A(Q_i,P_i)=0$.
Similarly, for $i\neq j$ in $\{1,\ldots,m\}$, we have $\Hom_A(Q_j,P_i)=0$, because $\sC_i$ and $\sC_j$ are orthogonal.
Finally, let $P=P_1\oplus\ldots\oplus P_m$, $Q=Q_1\oplus\ldots\oplus Q_m$, $N=\Hom_A(P,Q)$, and $D=\End_A(Q)$. We note that
$B^{(l)} = \End_A(P) = B_1^{(l)}\times\ldots\times B_m^{(l)}$, $D=D_1\times\ldots\times D_m$. Moreover, for each
$i\in\{1,\ldots,m\}$, we have an isomorphisms of algebras
\[ B(\sC_i) \cong \left[\begin{matrix} D_i & N_i \\ 0 & B_i^{(l)} \end{matrix}\right]. \]
Since $\Hom_A(Q,P)=0$, we obtain isomorphisms of algebras
\[ A \cong \End(A_A) \cong \left[\begin{matrix} D & N \\ 0 & B^{(l)} \end{matrix}\right] \cong B(\sC_1)\times\ldots\times B(\sC_m), \]
and this contradicts the indecomposablity of $A$. This finishes the proof of Theorem \ref{thm4}.
\end{proof}

\section{Generically tame algebras} \label{sec7}

The aim of this section is to prove Theorem \ref{thm10}, Corollary \ref{cor11}, Theorem \ref{thm12} using the known results
on generically tame quasitilted algebras.

Recall that, if $B$ is a~quasitilted algebra, then $\gldim B\leq 2$ and every module $X$ in $\ind B$ satisfies
$\pd_BX\leq 1$ or $\id_BX\leq 1$. In particular, we have a~well defined Euler characteristic
\[ \chi_B: K_0(B) \to {\Bbb Z} \]
such that $\chi_B(M) = |\End_B(M)| - |\Ext_B^1(M,M)|$ for any module $M$ in $\mo B$. We note that we may have $\chi_B$ indefinite
but with the well defined Euler characteristic $\chi_B(M)$ of any module $M$ in $\mo B$ satisfying
$\chi_B(M) = |\End_B(M)| - |\Ext_B^1(M,M)|\geq 0$.

Recall that following \cite{CB1, CB2} an algebra $A$ is called generically tame if for any positive integer $d$, there are
only finitely many isomorphism classes of generic modules of endolength $d$. It is known that if $A$ is generically tame
and $M$ is a~generic $A$-module then the endomorphism ring $\End_A(M)$ is a~$PI$ ring \cite[Section7]{CB2}.
Further, following \cite{CB2}, an algebra $A$ is called generically wild if there is a~generic $A$-module $M$ whose
endomorphism ring $\End_A(M)$ is not a~$PI$ ring. We also note that, if $H$ is an indecomposable hereditary algebra,
then $H$ is a~finite dimensional algebra over a~field (the center of $H$). It was already shown by Ringel \cite[Section 6]{Ri-01}
that a~hereditary algebra $H$ of Euclidean type admits a~unique generic module, and its endomorphism ring is a~$PI$ ring (by \cite[Corollary 6.12]{BGL}).
Further, Crawley-Boevey proved in \cite[Corollary 8.4]{CB2} that an indecomposable hereditary algebra is generically tame if and only if
$H$ is of Dynkin or Euclidean type.

Then, applying Propositions \ref{prop-tilted-1} and \ref{prop-tilted-2}, we obtain the following fact.
\begin{prop}\label{prop-71}
Let $H$ be an indecomposable hereditary algebra, $T$ a~tilting module in $\mo H$ without indecomposable preinjective (respectively, postprojective)
direct summands, and $B=\End_H(T)$ the associated tilted algebra.  Then the following statements are equivalent:
\begin{enumerate}
\renewcommand{\labelenumi}{\rm(\roman{enumi})}
\item $B$ is generically tame.
\item $B$ is generically finite.
\item $B$ is of Euclidean type.
\end{enumerate}
\end{prop}

We also recall the known characterizations of representation-infinite tilted algebras of Euclidean type.
\begin{prop}\label{prop-72}
Let $H$ be an indecomposable hereditary algebra, $T$ a~tilting module in $\mo H$ without indecomposable preinjective (respectively, postprojective)
direct summands, and $B=\End_H(T)$ the associated tilted algebra.  The following statements are equivalent:
\begin{enumerate}
\renewcommand{\labelenumi}{\rm(\roman{enumi})}
\item $B$ is of Euclidean type.
\item $\chi_B$ is positive semidefinite.
\item $\Gamma_B$ is almost periodic.
\item $\Sigma_B$ is acyclic.
\end{enumerate}
\end{prop}

It follows also from general theory that, if $B$ is a~tilted algebra of Euclidean type as above, then $\Gamma_B$
consists of a~postprojective component, a~preinjective component, and an infinite family ray tubes (respectively, coray tubes),
and all but finitely many of them are stable tubes of rank one.

Let $B$ be a~quasitilted algebra of canonical type with a~separating family of ray tubes (respectively, coray tubes).
Then we have three possibilities for $B$ with respect to behaviour of its Euler form $\chi_B$ (see \cite{LS}):
\begin{enumerate}
\renewcommand{\labelenumi}{\rm(\arabic{enumi})}
\item $\chi_B$ is positive semidefinite of corank $1$, and $B$ is a~tilted algebra of Euclidean type,
\item $\chi_B$ is positive semidefinite of corank $2$, and $B$ is a~tubular algebra in the sense of Ringel \cite{Ri1, Ri2}
(see also \cite{Le1}),
\item $\chi_B$ is indefinite.
\end{enumerate}

The following proposition follows from results established by Meltzer in \cite{Me}.
\begin{prop}\label{prop-73}
Let $B$ be a~quasitilted algebra of canonical type with a~separating family of ray tubes (respectively, coray tubes), and
assume that $\chi_B$ is indefinite. Then $B$ admits a~wild concealed quotient algebra $C$.
\end{prop}

The following fact was proved by Lenzing \cite{Le2}.
\begin{prop}\label{prop-74}
Let $B$ be a~tubular algebra. Then $B$ is generically infinite of polynomial growth.
\end{prop}

Summing up, we obtain the following fact.
\begin{prop}\label{prop-75}
Let $B$ be a~quasitilted algebra of canonical type with a~separating family of ray tubes (respectively, coray tubes).
Then the following statements are equivalent:
\begin{enumerate}
\renewcommand{\labelenumi}{\rm(\roman{enumi})}
\item $B$ is generically tame.
\item $B$ is generically of polynomial growth.
\item $\chi_B$ is positive semidefinite.
\item $B$ is a~tilted algebra of Euclidean type or a~tubular algebra.
\end{enumerate}
\end{prop}

Furthermore, applying again \cite{LP0, Me} and the known results on the module categories of tubular algebras,
we obtain the following proposition.
\begin{prop}\label{prop-76}
Let $B$ be a~quasitilted algebra of canonical type with a~separating family of ray tubes (respectively, coray tubes).
Then the following statements are equivalent:
\begin{enumerate}
\renewcommand{\labelenumi}{\rm(\roman{enumi})}
\item $B$ is a~tilted algebra of Euclidean type or a~tubular algebra.
\item $\Gamma_B$ is almost periodic.
\item $\Sigma_B$ is acyclic.
\end{enumerate}
\end{prop}

It follows also from general theory that, if $B$ is a~tubular algebra, then $\Gamma_B$ consists of a~postprojective component,
a~preinjective component and infinitely many ray tubes and coray tubes, and there are infinitely many stable tubes of rank at least two.

Finally, we recall also the following known fact.
\begin{prop}\label{prop-77}
Let $B$ be a~wild concealed algebra. Then there are infinitely many isomorphism classes of modules $M$ in $\ind B$ such that
$|\End_B(M)| < |\Ext_B^1(M,M)|$.
\end{prop}

Now Theorem \ref{thm10}, Corollary \ref{cor11}, and Theorem \ref{thm12} follow from Theorem \ref{thm7} and the results presented above.





\bibliographystyle{elsarticle-num}
\bibliography{<your-bib-database>}







\end{document}